\documentclass[12pt, reqno]{amsart}
\usepackage{amsmath, amsthm, amscd, amsfonts, amssymb, graphicx, color, relsize}
\usepackage{tikz-cd}
\usepackage[bookmarksnumbered, colorlinks, plainpages]{hyperref}
\newtheorem{theorem}{Theorem}[section]

\newtheorem{proposition}[theorem]{Proposition}

\newtheorem{corollary}[theorem]{Corollary}
\newtheorem{definition}[theorem]{Definition}

\newtheorem{remark}[theorem]{Remark}

\usepackage[bookmarksnumbered, colorlinks, plainpages]{hyperref}
\hypersetup{colorlinks=true,linkcolor=blue, anchorcolor=green, citecolor=cyan, urlcolor=red, filecolor=magenta, pdftoolbar=true}
\newcommand{\F}{\mathfrak{F}}
\newcommand{\J}{\mathfrak{J}}

\newcommand{\I}{\mathfrak{I}}

\usepackage{comment}
\excludecomment{mysection}

\begin{document}
	
\title[$K_0$-group of absolute matrix order unit spaces]{$K_0$-group of absolute Matrix order unit spaces}
\author{Anil Kumar Karn and Amit kumar}
	
\address{School of Mathematical Sciences, National Institute of Science Education and Research, HBNI, Bhubaneswar, P.O. - Jatni, District - Khurda, Odisha - 752050, India.}

\email{\textcolor[rgb]{0.00,0.00,0.84}{anilkarn@niser.ac.in, amit.kumar@niser.ac.in}}

	

\begin{abstract}
In this paper, we describe the Grothendieck group $K_0(V)$ of an absolute matrix order unit space $V$. For this purpose, we discuss the direct limit of absolute matrix order unit spaces. We show that $K_0$ is a functor from category of absolute matrix order unit spaces with morphisms as unital completely $\vert \cdot \vert$-preserving maps to category of abelian groups. We study order structure on $K_0(V)$ and prove that under certain condition $K_0(V)$ is an ordered abelian group. We also show that the functor $K_0$ is additive on orthogonal unital completely $\vert \cdot \vert$-preserving maps.
\end{abstract}

\thanks{The second author was financially supported by the Senior Research Fellowship of the University Grants Commission of India.}

\maketitle

\section{Introduction}

Operator $K$-theory is a primary component of the area of non-commutative topology. It also plays a prominent role in the structure theory of C$^*$-algebras. Due to its applications and appeal, operator $K$-theory has attracted a wide range of mathematicians. For a reference we may see, for example, \cite{B98, B06, RLL00} and references therein. 

In the recent years, the first author, in collaboration with several others, has been working on the order theoretic aspects of C$^*$-algebras. In \cite{K18}, he proposed the notion of absolute order unit spaces. Unital $JB$-algebras in general, and the self-adjoint parts of unital C$^*$-algebras in particulars, are classes of absolute order unit spaces. He further introduced a set of equivalent conditions which is necessary as well as sufficient for an absolute order unit space to be an $AM$-space (see \cite[Theorem 4.12]{K16}). In this context, an absolute order unit space may be considered as a model for non-commutative $AM$-spaces. 

In \cite{K19}, the authors introduced the matricial version of absolute order unit spaces, namely, the notion of absolute matrix order unit spaces and studied the absolute value preserving maps on these spaces. In \cite{PI19}, they studied various comparisons of order projections in an absolute matrix order unit space. The notion of order projections was introduced in \cite{K18} and order theoretically generalizes projections in a C$^*$-algebra. The main goal of this paper is initiate a study of $K$-theory in absolute matrix order unit spaces. 

Let us recall that the Grothendieck group $K_0(A)$ of a unital C$^*$-algebra $A$ is defined via the class of projections $\mathcal{P}_{\infty}(A)$ corresponding to the inductive (direct) limit $M_{\infty}(A)$ of the family of C$^*$-algebras $\lbrace M_n(A) \rbrace$. We follow the similar path. In order to discuss a suitable Grothendieck group corresponding to an absolute matrix order unit space, we describe a suiatable inductive (direct) limit. We spend the first part of this paper to describe the inductive limit of an absolute matrix order unit space. The idea of the present paper is based on the papers \cite{ER88, RKY05, RKY06, RKY08}. 

A matricially normed space $(V, \lbrace \Vert\cdot\Vert_n \rbrace)$ can be realized as the $\mathcal{F}$-bimodule $M_{\infty}(V)$ via the inductive (direct) limit of the sequence 
\[ V \subset M_2(V) \subset \dots \subset M_n(V) \subset \dots. \] 
This idea was suggested by B. E. Johnson to Effros and Ruan and was discussed in \cite{ER88}. In \cite{RKY05, RKY06, RKY08}, the first author along with others discussed the inductive (direct) limit of matrix ordered spaces. In this paper we describe the inductive (direct) limit of absolute matrix order unit spaces. 

In the other part of the paper, we develop $K_0$ of an absolute matrix order unit space. For this purpose, we extend the study of comparison of projections in an absolute matrix order unit space initiated in \cite{KK21}. We introduce the Grodendieck group $K_0(V)$ of the equivalence classes of projections in an absolute matrix order unit space $V$ as its $K_0$-group. We study an order structure in $K_0(V)$. We prove that $K_0$ is a functor from the category of absolute matrix order unit spaces with morphisms as unital completely absolute value preserving maps to category of abelian groups. We also define orthogonality of completely positive maps and prove that sum of two orthogonal completely absolute value preserving maps is again a completely absolute value preserving map. We further prove that $K_0$ is additive on orthogonal unital completely absolute value preserving maps. 

The development of the paper is as follows. In section 2, we have recalled inductive (direct) limit of matrix ordered spaces. In section 3, we have described the inductive (direct) limit of absolute matrix order unit spaces. In section 4, we develop $K_0$-group of an absolute matrix order unit space. In this section, we have also studied order structure in $K_0$-groups. In section 5, we have studied functorial nature of $K_0$. Further, we have proved that $K_0$ is additive on orthogonal unital completely absolute value preserving maps. 

\section{Inductive (direct) limit of matrix ordered spaces}
In this section, we recall the direct limit of matrix ordered spaces. 

\subsection{Matricial notions.}

Let $V$ be a complex vector space. We denote by $M_{m,n}(V)$ the vector space of all the $m \times n$ matrices $v=[v_{i,j}]$ with entries $v_{i,j} \in V$ and by $M_{m,n}$ the vector space of all the $m \times n$ matrices $a = [a_{i,j}]$ with entries $a_{i,j} \in \mathbb{C}$. We write $0_{m,n}$ for zero element in $M_{m,n}(V)$. For $m=n$, we write $0_{m,n} = 0_n$. We define $av = \begin{bmatrix} \displaystyle \sum_{k=1}^m a_{i,k}v_{k,j}\end{bmatrix}$ and $vb = \begin{bmatrix} \displaystyle \sum_{k=1}^n v_{i,k} b_{k,j} \end{bmatrix}$ for $a \in M_{r,m}, v \in M_{m,n}(V)$ and $b \in M_{n,s}$. We write 

\begin{center}
$v \oplus w = \begin{bmatrix} v & 0 \\ 0 & w\end{bmatrix}$ for $v \in M_{m,n}(V),w \in M_{r,s}(V).$
\end{center}

Here $0$ denotes suitable rectangular matrix of zero entries from $V$ \cite{ZJR88}.

Consider the family $\lbrace M_n \rbrace$. For each $n,m \in \mathbb{N}$ define $\sigma_{n,n+m}:M_n \to M_{n+m}$ given by $\sigma_{n,n+m}(a) = a \oplus 0_m$. Then $\sigma_{n,n+m}$ is a \emph{vector space isomorphism} with $$\sigma_{n,n+m}(ab)=\sigma_{n,n+m}(a)\sigma_{n,n+m}(b).$$

We observe that $\lbrace M_n, \sigma_{n,n+m},\mathbb{N}\rbrace$ is a direct system. Let $\F$ denote the set of all the $\infty \times \infty$ complex matrices having atmost finitely many non-zero entries. For each $n \in \mathbb{N}$, define $\sigma_n: M_n \to \F$ given by $\sigma_n(a)= a \oplus \mathfrak{0}$ for all $a \in M_n$, where $\mathfrak{0}$ denotes the zero element in $\F$. Then $\lbrace \F, \sigma_n\rbrace$ is the \emph{inductive limit} of $\lbrace M_n, \sigma_{n,n+m},\mathbb{N}\rbrace$. In fact, we have $$\F=\displaystyle \bigcup_{n=1}^\infty \sigma_n(M_n).$$

Let $\mathfrak{e}_{ij}$ denote the $\infty \times \infty$ matrix with $1$ at the $(i,j)$th entry and $0$ elsewhere. Then the collection $\lbrace \mathfrak{e}_{ij}\rbrace$ is called the set of matrix units in $\F$. We write $\I_n$ for $\displaystyle \sum_{i=1}^n \mathfrak{e}_{ii}.$ 

For $i,j,k,l \in \mathbb{N}$, we have $\mathfrak{e}_{ij} \mathfrak{e}_{kl}=\delta_{jk} \mathfrak{e}_{il}$ where

\[ 
   \delta_{jk} = \begin{cases}
    1 & \text{if $j=k$}  \\
    0 & \text{otherwise}.
   \end{cases}
\]

Note that for any $\mathfrak{a} \in \F$, there exist complex numbers $a_{ij}$ such that 
\begin{center}
$\mathfrak{a} = \displaystyle \sum_{i,j} a_{ij}\mathfrak{e}_{ij}$ (a finite sum).   
\end{center}
Thus $\F$ is an \emph{algebra}.

For $\mathfrak{a} = \displaystyle \sum_{i,j} a_{ij} \mathfrak{e}_{ij} \in \F$, we define $\mathfrak{a}^* = \displaystyle \sum_{i,j}\bar{a}_{ji}\mathfrak{e}_{ij} \in \F$. Then $\mathfrak{a} \longmapsto \mathfrak{a}^*$ is an \emph{involution}. In other words, $\F$ is a \emph{$\ast$-algebra}.

For the details please refer to \cite{ER88}. 

\subsection{Matrix ordered spaces.}

A complex vector space with an involution is called a \emph{$\ast$-vector space}. We write $V_{sa} = \lbrace v \in V: v = v^*\rbrace$. Then $V_{sa}$ is a real vector space \cite{CE77}. 

\begin{definition}\cite{CE77}
A \emph{matrix ordered space} is a $*$-vector space $V$ together with a sequence $\lbrace M_n(V)^+ \rbrace$ with $M_n(V)^+ \subset M_n(V)_{sa}$ for each $n \in \mathbb{N}$ satisfying the following conditions: 
\begin{enumerate}
	\item[(a)] $(M_n(V)_{sa}, M_n(V)^+)$ is a real ordered vector space, for each $n \in \mathbb{N}$; and  
	\item[(b)] $a^* v a \in M_m(V)^+$ for all $v \in M_n(V)^+$, $a \in M_{n,m}$ and $n ,m \in \mathbb{N}$. 
\end{enumerate} 
It is denoted by $(V, \lbrace M_n(V)^+ \rbrace)$. If, in addition, $e \in V^+$ is an order unit in $V_{sa}$ such that $V^+$ is proper and $M_n(V)^+$ is Archimedean for all $n \in \mathbb{N}$, then $V$ is called a \emph{matrix order unit space} and is denoted by $(V, \lbrace M_n(V)^+ \rbrace, e)$.
\end{definition}

\begin{proposition}
Let $(V, \lbrace M_n(V)^+\rbrace)$ be a matrix ordered space. 
\begin{enumerate}
\item[(1)]\cite[Proposition 1.8]{KV97}
\begin{enumerate}
\item[(a)] If $V^+$ is proper, then $M_n(V)^+$ is proper for all $n \in \mathbb{N}$.
\item[(b)] If $V^+$ is generating, then $M_n(V)^+$ is generating for all $n \in \mathbb{N}$.
\end{enumerate}
\item[(2)] \cite[Lemma 2.6]{KV97} If $e \in V^+$ is an order unit for $V_{sa}$. Then $e^n$ is an order unit for $M_n(V)_{sa}$ for all $n \in \mathbb{N}$ ( where $e^n : = e \oplus  \cdots \oplus e \in M_n(V)$).
\end{enumerate}
\end{proposition}

\begin{definition}\cite{ER88}\label{c}
Let $V$ be a complex vector space. Consider the family $\lbrace M_n(V)\rbrace.$ For each $n,m \in \mathbb{N}$, define $T_{n,n+m}:M_n(V) \to M_{n+m}(V)$ by $T_{n,n+m}(v)=v \oplus 0_m, 0_m \in M_m(V)$. Then $T_{n,n+m}$ is an injective homomorphism. Let $\lbrace \mathfrak{V},T_n\rbrace$ be the inductive limit of the directed family $\lbrace M_n(V),T_{n,n+m},\mathbb{N}\rbrace$ so that $T_n = T_{n+m} \circ T_{n,n+m}$ for all $m,n \in \mathbb{N}$. Then $\mathfrak{V}$ is an $\F$-bimodule. We shall call $\mathfrak{V}$ the \emph{matricial inductive limit or direct limit} of $V$.
\end{definition}

\begin{definition}\cite{ER88}
An $\F$-bimodule $\mathfrak{V}$ is said to be \emph{non-degenerate}, if for every $\mathfrak{v} \in \mathfrak{V}$ there exists  $n \in \mathbb{N}$ such that $\I_n \mathfrak{v} \I_n = \mathfrak{v}$. 
\end{definition}

The matricial inductive limit of a complex vector space may be characterized in the following sense:

\begin{theorem}\cite{ER88}
The matricial inductive limit of a complex vector space is a non-degenerate $\F$-bimodule. Conversely, let $\mathfrak{V}$ be a non-degenerate $\F$-bimodule. Put $V=\I_1 \mathfrak{V} \I_1$. Then $V$ is a complex vector space and $\mathfrak{V}$ is its matricial inductive limit in the sense of Definition \ref{c}. Moreover,
\begin{enumerate}
\item[(a)] $T_n(M_n(V)) = \I_n \mathfrak{V} \I_n$.
\item[(b)] $\mathfrak{V}=\displaystyle\bigcup_{n=1}^\infty T_n(M_n(V))$.
\end{enumerate}
\end{theorem}

Let $\mathfrak{V}$ be a non-degenerate $\mathfrak{F}$-bimodule. Also let $\mathfrak{v} \in \mathfrak{V}$ and $\alpha \in \mathbb{C}$. We write $\alpha \mathfrak{v} = (\alpha \I_n) \mathfrak{v}$ for some $n \in \mathbb{N}$ with $\I_n \mathfrak{v} \I_n = \mathfrak{v}$. Then $\alpha \mathfrak{v}$ is well-defined. Thus $\mathfrak{V}$ is a complex vector space.

Now we recall the notion of $\mathfrak{F}$-bimodule norm on a non-degenerate $\mathfrak{F}$-bimodule in the following sense:

\begin{definition}\cite[Definition 1.4]{RKY06}
Let $\mathfrak{V}$ be a non-degenerate $\mathfrak{F}$-bimodule. Let $\Vert \cdot \Vert$ be a norm on $\mathfrak{V}$. Then we say that $\Vert \cdot \Vert$ is an $\mathfrak{F}$-bimodule norm on $\mathfrak{V}$, if $\Vert \mathfrak{a} \mathfrak{v} \mathfrak{b}\Vert \le \Vert \mathfrak{a} \Vert \Vert \mathfrak{v} \Vert \Vert \mathfrak{b} \Vert$ for any $\mathfrak{a}, \mathfrak{b} \in \mathfrak{F}$ and $\mathfrak{v} \in \mathfrak{V}$.
\end{definition}

Next, we describe the order theoretic aspect.

\begin{definition}\cite[Definition 2.6]{RKY05}
Let $\mathfrak{V}$ be an $\F$-bimodule and let $*:\mathfrak{V}\longrightarrow \mathfrak{V}$  be a map satisfying the following conditions:
\begin{enumerate}
\item[1.]$(\mathfrak{u}+\mathfrak{v})^*=\mathfrak{u}^*+\mathfrak{v}^*$ for all $\mathfrak{u},\mathfrak{v}\in \mathfrak{V};$
\item[2.]$(\mathfrak{a} \mathfrak{v})^*=\mathfrak{v}^* \mathfrak{a}^*,~(\mathfrak{v} \mathfrak{a})^*=\mathfrak{a}^* \mathfrak{v}^*$ for all $\mathfrak{v} \in \mathfrak{V},\mathfrak{a} \in \F;$
\item[3.]$(\mathfrak{v}^*)^*=\mathfrak{v}$ for all $\mathfrak{v}\in \mathfrak{V}.$ 
\end{enumerate}
Then $*$ is called an \emph{involution} on $\mathfrak{V}$ and in this case $\mathfrak{V}$ is called a \emph{$*$-$\F$-bimodule}. We put $\mathfrak{V}_{sa}=\lbrace \mathfrak{v} \in \mathfrak{V}: \mathfrak{v}^*=\mathfrak{v} \rbrace.$
\end{definition}

\begin{definition}\cite[Definition 3.2]{RKY05} \label{w46}
Let $\mathfrak{V}$ be a $*$-$\F$-bimodule and let $\mathfrak{V}^+\subset \mathfrak{V}_{sa}$ satisfying the following conditions: 
\begin{enumerate}
\item[1.]$\mathfrak{u} + \mathfrak{v} \in \mathfrak{V}^+$ for all $\mathfrak{u},\mathfrak{v} \in \mathfrak{V}^+;$
\item[2.]$\mathfrak{a}^* \mathfrak{v} \mathfrak{a} \in \mathfrak{V}^+$ for all $\mathfrak{v} \in \mathfrak{V}^+,\mathfrak{a} \in \F.$
\end{enumerate}
Then $\mathfrak{V}^+$ is called a \emph{bimodule cone}  and $(\mathfrak{V},\mathfrak{V}^+)$ is called an \emph{ordered $\F$-bimodule}. 
\end{definition}

\begin{definition}\cite{RKY05}
Let $(\mathfrak{V},\mathfrak{V}^+)$ be an ordered $\F$-bimodule. We say that $\mathfrak{V}^+$ is \emph{proper}, if $\mathfrak{V} \cap (-\mathfrak{V}^+)=\lbrace 0 \rbrace$ and \emph{generating}, if given $\mathfrak{v} \in \mathfrak{V}$ there exist $\mathfrak{v}_0,\mathfrak{v}_1,\mathfrak{v}_2,\mathfrak{v}_3 \in \mathfrak{V}^+$ such that $\mathfrak{v} = \displaystyle \sum_{k=0}^3 i^k \mathfrak{v}_k$, where $i^2=-1.$ We say that $\mathfrak{V}^+$ is \emph{Archimedean}, if for any $\mathfrak{v} \in \mathfrak{V}_{sa}$ with $k \mathfrak{u} + \mathfrak{v} \in \mathfrak{V}^+$ for a fixed $\mathfrak{u} \in \mathfrak{V}^+$ and all positive real numbers $k$, we have $\mathfrak{v} \in \mathfrak{V}^+$.
\end{definition}

Next theorem characterize the direct limit of a matrix ordered space as a non-degenerate ordered $\mathfrak{F}$-bimodule.

\begin{theorem}\cite[Theorem 3.4]{RKY05}\label{w16}
Let $(V,\lbrace M_n(V)^+\rbrace)$ be a matrix ordered space and let $\mathfrak{V}$ be the matricial inductive limit of $V.$ Then $T_{n,n+m}$ and $T_n$ are positive maps. Put $\mathfrak{V}^+ = \displaystyle \bigcup_{n=1}^{\infty} T_n(M_n(V)^+).$ Then $(\mathfrak{V},\mathfrak{V}^+)$ is a non-degenerate ordered $\F$-bimodule. Conversely, let $(\mathfrak{V},\mathfrak{V}^+)$ be a non-degenerate ordered $\F$-bimodule and put $V=\I_1\mathfrak{V} \I_1$. Define $T_n : M_n(V) \to \mathfrak{V}$ given by $$T_n([v_{i,j}])=\displaystyle \sum_{i,j=1}^n \mathfrak{e}_{i,1} v_{i,j} \mathfrak{e}_{1,j}$$ for all $[v_{i,j}] \in M_n(V).$ Then $T_n$ is an injective homomorphism such that $T_n(M_n(V))=\I_n \mathfrak{V} \I_n$ for all $n \in \mathbb{N}$. Set $M_n(V)^+ = T_n^{-1}(\I_n \mathfrak{V}^+\I_n)$ for all $n \in \mathbb{N}$. Then $(V,\lbrace M_n(V)^+\rbrace)$ is a matrix ordered space and $\mathfrak{V}$ is its matricial inductive limit with $\mathfrak{V}^+=\displaystyle \bigcup_{n=1}^{\infty} T_n(M_n(V)^+).$
\end{theorem}

\begin{remark}\label{h} 
	Let $(\mathfrak{V},\mathfrak{V}^+)$ be a non-degenerate ordered $\F$-bimodule and let $(V, \lbrace M_n(V)^+\rbrace)$ be the corresponding matrix ordered space as in Theorem \ref{w16}. Then 
\begin{enumerate}
\item[(1)] $\mathfrak{V}^+$ is proper if and only if $V^+$ is proper \cite[Theorem 3.9]{RKY05}. 
\item[(2)] $\mathfrak{V}^+$ is generating if and only if $V^+$ is generating \cite[Theorem 3.12]{RKY05}.
\item[(3)] $\mathfrak{V}^+$ is Archimedean if and only if $M_n(V)^+$ is Archimedean for each $n \in \mathbb{N}$.
\end{enumerate}
\end{remark}

\section{Inductive (direct) limit of absolute matrix order unit spaces}

In this section, we describe the direct limits of absolutely matrix ordered spaces and absolute matrix order unit spaces. First, we recall the following notion introduced in \cite{K19}. 

\begin{definition}\cite[Definition 4.1]{K19}\label{152}
Let $(V, \lbrace \ M_n(V)^+ \rbrace)$ be a matrix ordered space and assume that  $\vert\cdot\vert_{m,n}: M_{m,n}(V) \to M_n(V)^+$ for $m, n \in \mathbb{N}$. Let us write $\vert\cdot\vert_{n,n} = \vert\cdot\vert_n$ for every $n \in \mathbb{N}$. Then $\left(V, \lbrace M_n(V)^+ \rbrace, \lbrace \vert\cdot\vert_{m,n} \rbrace \right)$ is called an \emph{absolutely matrix ordered space}, if it satisfies the following conditions: 
\begin{enumerate}
\item[$1.$] For all $n \in \mathbb{N}$, $(M_n(V)_{sa}, M_n(V)^+, \vert\cdot\vert_n)$ is an absolutely ordered space;
\item[$2.$] For $v \in M_{m,n}(V), \alpha \in M_{r,m}$ and $\beta \in M_{n,s},$ we have
$$\vert \alpha v \beta \vert_{r,s} \leq \| \alpha \| \vert \vert v \vert_{m,n} \beta \vert_{n,s};$$
\item[$3.$] For $v \in M_{m,n}(V)$ and $w \in M_{r,s}(V),$ we have
$$\vert v \oplus w\vert_{m+r,n+s} = \vert v \vert_{m,n} \oplus \vert w \vert_{r,s}.$$
Here $v \oplus w := \begin{bmatrix} v & 0 \\ 0 & w \end{bmatrix}$.
\end{enumerate} 
\end{definition}

\begin{proposition}\cite[Proposition 4.2]{K19}\label{151}
Let $(V, \lbrace M_n(V)^+ \rbrace, \lbrace \vert\cdot\vert_{m,n} \rbrace)$ be an absolutely matrix ordered space.  
\begin{enumerate}
\item[$1.$] If $\alpha \in M_{r,m}$ is an isometry i.e. $\alpha^* \alpha = I_m,$ then $\vert \alpha v \vert_{r,n} = \vert v \vert_{m,n}$ for any $v \in M_{m,n}(V).$
\item[$2.$] If $v \in M_{m,n}(V),$ then $\left\vert \begin{bmatrix} 0_m & v \\ v^* & 0_n \end{bmatrix} \right\vert_{m+n} = \vert v^* \vert_{n,m} \oplus \vert v \vert_{m,n}.$
\item[$3.$] $\begin{bmatrix} \vert v^* \vert_{n,m} & v \\ v^* & \vert v \vert_{m,n} \end{bmatrix} \in M_{m+n}(V)^+$ for any $v \in M_{m,n}(V).$
\item[$4.$] $\vert v \vert_{m,n} = \left\vert \begin{bmatrix} v \\ 0 \end{bmatrix} \right\vert_{m+r,n}$ for any $v \in M_{m,n}(V)$ and $r \in \mathbb{N}.$
\item[$5.$] $\vert v \vert_{m,n} \oplus 0_s = \left\vert \begin{bmatrix} v & 0 \end{bmatrix} \right\vert_{m,n+s}$ for any $v \in M_{m,n}(V)$ and $s \in \mathbb{N}.$
\end{enumerate} 
\end{proposition}

\begin{definition}\cite[Definition 4.3]{K19}
Let $(V, \lbrace M_n(V)^+ \rbrace, e)$ be a matrix order unit space such that 
\begin{enumerate}
\item[(a)] $\left(V, \lbrace M_n(V)^+ \rbrace, \lbrace \vert \cdot \vert_{m,n} \rbrace \right)$ is an absolutely matrix ordered space; and
\item[(b)]$\perp = \perp_{\infty}^a$ on $M_n(V)^+$ for all $n \in \mathbb{N}.$ 
\end{enumerate}
Then $(V, \lbrace M_n(V)^+ \rbrace, \lbrace \vert\cdot\vert_{m,n} \rbrace, e)$ is called an \emph{absolute matrix order unit space}. 
\end{definition}

Next, we define the order of an element in a non-degenerate $\F$-bimodule which we need in describing the direct limit of absolute matrix ordered space.

\begin{definition}
Let $\mathfrak{V}$ be a non-degenerate $\F$-bimodule and let $\mathfrak{v} \in \mathfrak{V}.$ Then the smallest $n\in \mathbb{N}$ such that $\I_n \mathfrak{v} \I_n = \mathfrak{v}$ is called the \emph{order} of $\mathfrak{v}$ in $\mathfrak{V}.$ We denote the order of $\mathfrak{v}$ in $\mathfrak{V}$ by $o(\mathfrak{v}).$ 
\end{definition}

\begin{proposition}\label{150}
Let $(\mathfrak{V},\mathfrak{V}^{+})$ be a non-degenerate ordered $\F$-bimodule and assume that $\mathfrak{V}^+$ is proper and Archimedean. Then for $\mathfrak{u},\mathfrak{v} \in \mathfrak{V}^+$ with $\mathfrak{u} \le \mathfrak{v}$, we have $o(\mathfrak{u}) \le o(\mathfrak{v})$.
\end{proposition}

\begin{proof}
If possible, assume that $n=o(\mathfrak{u}) > o(\mathfrak{v})$. Then $\mathfrak{e}_{1,n} \mathfrak{v} \mathfrak{e}_{n,1}=0$ so that $0 \le \mathfrak{e}_{1,n} \mathfrak{u} \mathfrak{e}_{n,1} \le \mathfrak{e}_{1,n} \mathfrak{v} \mathfrak{e}_{n,1} = 0$. 
As $\mathfrak{V}^+$ is proper, we get $\mathfrak{e}_{1,n} \mathfrak{u} \mathfrak{e}_{n,1}=0$. Since $o(\mathfrak{u})=n$, there exists $1 \le j < n$ such that 
$\mathfrak{e}_{1,j} \mathfrak{u} \mathfrak{e}_{n,1} \neq 0$. Let $\lambda \in \mathbb{C}$ and put $\mathfrak{a} = (\lambda \mathfrak{e}_{j,1}+\mathfrak{e}_{n,1})$. Then 
$\mathfrak{a}^* \mathfrak{u} \mathfrak{a} \in \mathfrak{V}^+$ so that 
$$\vert \lambda \vert^2 \mathfrak{e}_{1,j} \mathfrak{u} \mathfrak{e}_{j,1} + \bar{\lambda} \mathfrak{e}_{1,j} \mathfrak{u} \mathfrak{e}_{n,1}+\lambda \mathfrak{e}_{1,n} \mathfrak{u} \mathfrak{e}_{j,1} \in \mathfrak{V}^+$$ 
for all $\lambda \in \mathbb{C}$. Thus 
$$k \mathfrak{e}_{1,j} \mathfrak{u} \mathfrak{e}_{j,1} + \bar{\alpha} \mathfrak{e}_{1,j} \mathfrak{u} \mathfrak{e}_{n,1}+ \alpha \mathfrak{e}_{1,n} \mathfrak{u} \mathfrak{e}_{j,1} \in \mathfrak{V}^+$$ 
for all $k > 0$ and $\alpha \in \mathbb{C}$ with $\vert \alpha \vert =1$. We have $\mathfrak{e}_{1,j} \mathfrak{u} \mathfrak{e}_{j,1} \in \mathfrak{V}^+$ and $\mathfrak{V}^+$ is Archimedean, we conclude that  
$\bar{\alpha} \mathfrak{e}_{1,j} \mathfrak{u} \mathfrak{e}_{n,1}+ \alpha \mathfrak{e}_{1,n} \mathfrak{u} \mathfrak{e}_{j,1} \in \mathfrak{V}^+$ for all $\alpha \in \mathbb{C}$ with $\vert \alpha \vert =1$. Since $\mathfrak{V}^+$ is proper, a standard use of $\alpha \in \mathbb{C}$ with $\vert \alpha \vert = 1$ yields that $\mathfrak{e}_{1,j} \mathfrak{u} \mathfrak{e}_{n,1} = 0$ which is a contradiction. Hence $o(\mathfrak{u}) \le o(\mathfrak{v})$.
\end{proof}

\begin{definition}Let $\mathfrak{V}$ be a non-degenerate $\F$-bimodule and let $\mathfrak{u},\mathfrak{v} \in \mathfrak{V}.$ Then $\mathfrak{u}$ and $\mathfrak{v}$ are said to be \emph{$\F$-independent} in $\mathfrak{V}$, if there exist $I,J\subset\mathbb{N}$ with $I\cap J=\phi$ such that $\mathfrak{r} \mathfrak{u} \mathfrak{r} = \mathfrak{u}$ and $\mathfrak{s} \mathfrak{v} \mathfrak{s} = \mathfrak{v}$ where $\mathfrak{r}=\displaystyle\sum_{i\in I} \mathfrak{e}_{i,i}$ and $\mathfrak{s}=\displaystyle\sum_{j\in J} \mathfrak{e}_{j,j}.$ 
\end{definition}

\textbf{A notation:} Let $(\mathfrak{V},\mathfrak{V}^{+})$ be a non-degenerate ordered $\F$-bimodule and let $\mathfrak{v}_1,\mathfrak{v}_2 \in \mathfrak{V}^+$ with $\I_n \mathfrak{v}_1 \I_n = \mathfrak{v}_1$ and $\I_n \mathfrak{v}_2 \I_n = \mathfrak{v}_2$ for some $n \in \mathbb{N}$. Put $\J_n = \displaystyle \sum_{i=1}^n \mathfrak{e}_{i,n+i}$. Then we write $(\mathfrak{v}_1,\mathfrak{v}_2)_n^+ = \mathfrak{v}_1 + \J_n^* \mathfrak{v}_2 \J_n$ and $sa_n(\mathfrak{v})=\I_n \mathfrak{v} \J_n + \J_n^* \mathfrak{v}^* \I_n$ for some $\mathfrak{v} \in \mathfrak{V}$ with $\I_n \mathfrak{v} \I_n = \mathfrak{v}.$ 

Let $(\mathfrak{V},\mathfrak{V}^+)$ be the matricial inductive limit of matrix ordered space $(V, \lbrace M_n(V)^+\rbrace)$. Put $v_1 = T_n^{-1}(\mathfrak{v}_1),v_2 = T_n^{-1}(\mathfrak{v}_2)$ and $v=T_n^{-1}(\mathfrak{v})$. Then $\begin{bmatrix}v_1 & 0 \\ 0 & v_2\end{bmatrix} = T_{2n}^{-1}((\mathfrak{v}_1,\mathfrak{v}_2)_n^+)$ and $\begin{bmatrix} 0 & v \\ v^* & 0 \end{bmatrix} = T_{2n}^{-1}(sa_n(\mathfrak{v}))$. Hence $\begin{bmatrix} v_1 & \pm v \\ \pm v^* & v_2 \end{bmatrix} = T_{2n}^{-1}((\mathfrak{v}_1,\mathfrak{v}_2)_n^+ \pm sa_n(\mathfrak{v}))$.

\begin{definition}\label{18}
Let $(\mathfrak{V},\mathfrak{V}^{+})$ be a non-degenerate ordered $\F$-bimodule. Let $\vert \cdot \vert:\mathfrak{V}\longrightarrow \mathfrak{V}^{+}$ be a map satisfying the following conditions:
\begin{enumerate}
\item[1.]$o(\vert \mathfrak{v}\vert) \le o(\mathfrak{v});$
\item[2.]$\vert \mathfrak{v} \vert=\mathfrak{v}$ if $\mathfrak{v} \in \mathfrak{V}^{+};$
\item[3.]${(\vert \mathfrak{v}^*\vert,\vert \mathfrak{v}\vert)_{o(\mathfrak{v})}}^{+}+sa_{o(\mathfrak{v})}(\mathfrak{v})\in \mathfrak{V}^{+}$ for all  $\mathfrak{v} \in \mathfrak{V};$
\item[4.]For $\mathfrak{v} \in \mathfrak{V}$ and $\mathfrak{a} \in \F$
\begin{center}
${\vert\mathfrak{a} \mathfrak{v} \mathfrak{b} \vert} \le \Vert \mathfrak{a} \Vert \vert \vert \mathfrak{v} \vert \mathfrak{b} \vert;$
\end{center}
\item[5.]For $\mathfrak{u}$ and $\mathfrak{v}~\F$-independent elements in $(\mathfrak{V},\mathfrak{V}^{+})$, we have
\begin{center}
${\vert \mathfrak{u} + \mathfrak{v}\vert}={\vert \mathfrak{u} \vert}+ \vert \mathfrak{v} \vert;$
\end{center} 
\item[6.]For $\mathfrak{u},\mathfrak{v}$ and $\mathfrak{w} \in \mathfrak{V}^{+}$ with $\vert \mathfrak{u} - \mathfrak{v} \vert=\mathfrak{u}+\mathfrak{v}$ and $\mathfrak{0} \le \mathfrak{w} \le \mathfrak{v},$ we have $\vert \mathfrak{u}-\mathfrak{w}\vert=\mathfrak{u}+\mathfrak{w};$ 
\item[7.]For $\mathfrak{u},\mathfrak{v}$ and $\mathfrak{w} \in \mathfrak{V}^{+}$ with $\vert \mathfrak{u} - \mathfrak{v} \vert=\mathfrak{u}+\mathfrak{v}$ and $\vert \mathfrak{u}-\mathfrak{w} \vert=\mathfrak{u}+\mathfrak{w},$ we have $\vert \mathfrak{u}-\vert\mathfrak{v}\pm\mathfrak{w}\vert\vert=\mathfrak{u}+\vert \mathfrak{v} \pm \mathfrak{w} \vert.$
\end{enumerate}
Then $(\mathfrak{V},\mathfrak{V}^{+},\vert \cdot \vert)$ is said to be a \emph{non-degenerate absolutely ordered $\F$-bimodule}.
\end{definition}

\begin{proposition}
Let $(\mathfrak{V},\mathfrak{V}^{+},\vert \cdot \vert)$ be a non-degenerate absolutely ordered $\F$-bimodule. Then $\vert \mathfrak{a} \mathfrak{v} \vert = \vert \mathfrak{v} \vert$ for any $\mathfrak{v} \in \mathfrak{V}$ and $\mathfrak{a} \in \F$ with $\mathfrak{a}^*\mathfrak{a} = \I_{o(\mathfrak{v})}$.
\end{proposition}

\begin{proof}
Let $\mathfrak{a} \in \F$ with $\mathfrak{a}^*\mathfrak{a} = \I_{o(\mathfrak{v})}$. Then, by Definition \ref{18}(4), we get that
$$\vert \mathfrak{a} \mathfrak{v} \vert \le \| \mathfrak{a} \| \vert \mathfrak{v} \vert = \vert \mathfrak{a}^* \mathfrak{a} \mathfrak{v} \vert \le \| \mathfrak{a}^* \| \vert \mathfrak{a} \mathfrak{v} \vert = \vert \mathfrak{a} \mathfrak{v} \vert .$$ 
so that $\vert \mathfrak{a} \mathfrak{v} \vert = \vert \mathfrak{v} \vert.$
\end{proof}

\begin{remark}
Let $(\mathfrak{V},\mathfrak{V}^{+},\vert \cdot \vert)$ be a non-degenerate absolutely ordered $\F$-bimodule. Then
\begin{enumerate}
\item[(a)] $\vert \alpha \mathfrak{v}\vert = \vert \alpha \vert \vert \mathfrak{v}\vert$ for all $\mathfrak{v} \in \mathfrak{V}$ and $a \in \mathbb{C}$. To see this, let $\mathfrak{v} \in \mathfrak{V}$ and $\alpha \in \mathbb{C}$. Then $\alpha \mathfrak{v} = (\alpha \I_{o(\mathfrak{v})}) \mathfrak{v}$. Thus by \ref{18}(4), we have $\vert \alpha \mathfrak{v}\vert = \vert \alpha \vert \vert \mathfrak{v}\vert$. 
\item[(b)] $\mathfrak{V}^+$ is proper. To verify this, let $\pm \mathfrak{v} \in \mathfrak{V}^+.$ Then by \ref{18}(2) and by (a), we have $\mathfrak{v} = \vert \mathfrak{v}\vert = \vert -\mathfrak{v}\vert = -\mathfrak{v}$ so that $\mathfrak{v}=0.$
\item[(c)] $\mathfrak{V}^+$ is also generating. To see this, let $\mathfrak{v} \in \mathfrak{V}_{sa}.$ Then by \ref{18}(3), we get that $\vert \mathfrak{v}\vert \pm \mathfrak{v} \in \mathfrak{V}^+$ and hence $$\mathfrak{v} = \frac{1}{2}\left( (\vert \mathfrak{v} \vert + \mathfrak{v}) - (\vert \mathfrak{v} \vert - \mathfrak{v}) \right) \in \mathfrak{V}^+ - \mathfrak{V}^+.$$ 
\item[(d)] Let $\mathfrak{u}, \mathfrak{v} \in \mathfrak{V}_{sa}$ be such that $\vert \mathfrak{u} - \mathfrak{v} \vert = \mathfrak{u} + \mathfrak{v}$. Then $\mathfrak{u}, \mathfrak{v} \in \mathfrak{V}^+$. For such a pair $\mathfrak{u}, \mathfrak{v} \in \mathfrak{V}^+$, we shall say that $\mathfrak{u}$ is \emph{orthogonal} to $\mathfrak{v}$ and denote it by $\mathfrak{u} \perp \mathfrak{v}$.
\item[(e)] For $\mathfrak{v} \in \mathfrak{V}_{sa},$ we write $\mathfrak{v}^+ := \frac{1}{2}(\vert \mathfrak{v} \vert + \mathfrak{v})$ and $\mathfrak{v}^- := \frac{1}{2}(\vert \mathfrak{v} \vert - \mathfrak{v})$. Then $\mathfrak{v}^+ \perp \mathfrak{v}^-, \mathfrak{v} = \mathfrak{v}^+ - \mathfrak{v}^-$ and $\vert \mathfrak{v} \vert = \mathfrak{v}^+ + \mathfrak{v}^-$. This decomposition is unique in the following sense: If $\mathfrak{v} = \mathfrak{v}_1 - \mathfrak{v}_2$ with $\mathfrak{v}_1 \perp \mathfrak{v}_2$, then $\mathfrak{v}_1 = \mathfrak{v}^+$ and $\mathfrak{v}_2 = \mathfrak{v}^-$. In other words, every element in $\mathfrak{V}_{sa}$ has a \emph{unique orthogonal decomposition} in $\mathfrak{V}^+$. 
\end{enumerate}
\end{remark}

\begin{definition}
An element $\mathfrak{a} \in \F$ is said to be a \emph{local unitary}, if $\mathfrak{a}^*\mathfrak{a}=\I_{o(\mathfrak{a})}=\mathfrak{a}\mathfrak{a}^*.$
\end{definition}

\begin{proposition}\label{19}
Let $(\mathfrak{V},\mathfrak{V}^{+},\vert \cdot \vert)$ be a non-degenerate absolutely ordered $\F$-bimodule and let $\mathfrak{v} \in \mathfrak{V}$. Then $\vert\mathfrak{a}^*\mathfrak{v}\mathfrak{a}\vert=\mathfrak{a}^*\vert\mathfrak{v}\vert\mathfrak{a}$ for every local unitary element $\mathfrak{a} \in \F$ with $o(\mathfrak{v}) \le o(\mathfrak{a})$.
\end{proposition}

\begin{proof}
Let $\mathfrak{a} \in \F$ be a local unitary with $o(\mathfrak{v}) \le o(\mathfrak{a}).$ Since $o(\vert \mathfrak{v} \vert) \le o(\mathfrak{v})$, by Definition \ref{18}(4), we have

\begin{eqnarray*}
\vert\mathfrak{a}^*\mathfrak{v}\mathfrak{a}\vert &\le & \|\mathfrak{a}^*\|\vert \vert\mathfrak{v}\vert\mathfrak{a}\vert \\
&=&\vert \vert\mathfrak{v}\vert\mathfrak{a}\vert \\
&=& \vert\mathfrak{a}(\mathfrak{a}^*\vert\mathfrak{v}\vert\mathfrak{a})\vert \\
&\le & \|\mathfrak{a}\|\vert\mathfrak{a}^*\vert\mathfrak{v}\vert\mathfrak{a}\vert \\
&=& \mathfrak{a}^*\vert\mathfrak{v}\vert\mathfrak{a}.
\end{eqnarray*}
Similarly $\vert \mathfrak{a} \mathfrak{v} \mathfrak{a}^*\vert \le \mathfrak{a} \vert \mathfrak{v} \vert \mathfrak{a}^*$. 
Now $o(\mathfrak{v}) \le o(\mathfrak{a})$, thus replacing $\mathfrak{v}$ by $\mathfrak{a}^*\mathfrak{v}\mathfrak{a},$ we get $\vert\mathfrak{v}\vert\le \mathfrak{a}\vert\mathfrak{a}^*\mathfrak{v}\mathfrak{a}\vert\mathfrak{a}^*$ so that $\mathfrak{a}^*\vert\mathfrak{v}\vert\mathfrak{a}\le\vert\mathfrak{a}^*\mathfrak{v}\mathfrak{a}\vert.$ As $\mathfrak{V}^{+}$ is proper, we conclude that $\vert\mathfrak{a}^*\mathfrak{v}\mathfrak{a}\vert=\mathfrak{a}^*\vert\mathfrak{v}\vert\mathfrak{a}.$
\end{proof}

Now, in next two Theorems \ref{26} and \ref{z}, we generalize Theorem \ref{w16} in the context of absolutely matrix ordered space.   
 
\begin{theorem}\label{26}
Let $(V,\lbrace{M_{n}(V)}^{+}\rbrace,\lbrace {\vert \cdot \vert}_{n}\rbrace)$ be an absolutely matrix ordered space and let $(\mathfrak{V},\mathfrak{V}^{+})$ be the matricial inductive limit of $(V,\lbrace{M_{n}(V)}^{+}\rbrace).$ 
For $\mathfrak{v} \in \mathfrak{V}$, we define $\vert \mathfrak{v} \vert: = T_{o(\mathfrak{v})}(\vert T_{o(\mathfrak{v})}^{-1}(\mathfrak{v})\vert)$ and consider the map $\vert \cdot \vert : \mathfrak{V} \to \mathfrak{V}^+$ given by $\mathfrak{v} \longmapsto \vert \mathfrak{v} \vert$ for all $\mathfrak{v} \in \mathfrak{V}$. Then 
$(\mathfrak{V},\mathfrak{V}^{+},\vert \cdot \vert)$ is a non-degenerate absolutely ordered $\F$-bimodule.
\end{theorem}
\begin{proof}
It is routine to verify all the conditions in definition \ref{18} except (5).

Now, let $\mathfrak{u}$ and $\mathfrak{v}$ be $\F$-independent elements in $(\mathfrak{V},\mathfrak{V}^{+}).$ There exist $I,J\subset\mathbb{N}$ with $I\cap J=\phi$ such that $\mathfrak{r} \mathfrak{u} \mathfrak{r} = \mathfrak{u}$ and $\mathfrak{s} \mathfrak{v} \mathfrak{s} = \mathfrak{v}$ where $\mathfrak{r}=\displaystyle\sum_{i\in I}\mathfrak{e}_{i,i}$ and $\mathfrak{s}=\displaystyle\sum_{j\in J}\mathfrak{e}_{j,j}.$ 
Let 
$$I=\lbrace i_1<i_2< \cdots <i_{k-1}<i_{k}\rbrace,$$ 
$$J=\lbrace j_1<j_2< \cdots <j_{l-1}<j_{l}\rbrace$$ 
and put $n=\max\lbrace i_k,j_l\rbrace.$ Then $I\cup J\subset \lbrace 1,2,3,\cdots,n\rbrace$. Let us write 
$$G :=\lbrace 1,2,3,\cdots,n\rbrace \backslash (I\cup J)=\lbrace r_1,r_2,r_3,\cdots,r_{m-1},r_{m}\rbrace$$ 
so that $k+l+m=n.$ Put $$\mathfrak{a}=\sum_{t=1}^{k}\mathfrak{e}_{i_t,t},~\mathfrak{b}=\sum_{t=1}^{l}\mathfrak{e}_{j_t,t+k},\mathfrak{c}=\sum_{t=1}^{m}\mathfrak{e}_{r_t,t+k+l}$$ and $\mathfrak{d}=\mathfrak{a}+\mathfrak{b}+\mathfrak{c}.$ It is routine to verify that $\mathfrak{a}^*\mathfrak{a}=\I_{k},\mathfrak{a}\mathfrak{a}^*=\mathfrak{r},\mathfrak{b}^*\mathfrak{b}=0_{k}\oplus \I_{l},\mathfrak{b}\mathfrak{b}^*=\mathfrak{s},\mathfrak{c}^*\mathfrak{c}=0_{k+l}\oplus \I_{m}$, $\mathfrak{a}^*\mathfrak{b} =\mathfrak{a}\mathfrak{b}^* =\mathfrak{a}^* \mathfrak{c} =\mathfrak{a} \mathfrak{c}^* =\mathfrak{b}^* \mathfrak{c} =\mathfrak{b}\mathfrak{c}^*=\mathfrak{0}$ and $\mathfrak{d}^*\mathfrak{d}=\I_{n}=\mathfrak{d}\mathfrak{d}^*.$ Put $\mathfrak{u}_1=\mathfrak{a}^*\mathfrak{u}\mathfrak{a}$ and $\mathfrak{v}_1=\mathfrak{b}^*\mathfrak{v}\mathfrak{b}$. Then $\I_k \mathfrak{u}_1 \I_k=\mathfrak{u}_1$ and $(0_k\oplus \I_l)\mathfrak{v}_1 (0_k\oplus \I_l)=\mathfrak{v}_1.$ 

Let $T_{k}^{-1}(\mathfrak{u}_1) = u$ and $T_{k+l}^{-1}(\mathfrak{v}_1) = v_1$. Then $u \in M_k(V)$ and $v_1 \in M_{k+l}(V)$. Also then $$v_1 = T_{k+l}^{-1}((0_k \oplus \I_l) \mathfrak{v}_1 (0_k \oplus \I_l)) = (0_k \oplus I_l) v_1 (0_k \oplus I_l)$$ so that $v_1 = 0_k \oplus v$ for some $v \in M_l(V)$. Now $T_{k+l}^{-1}(\mathfrak{u}_1) = u \oplus 0_l$ and $T_{k+l}^{-1}(\mathfrak{v}_1) = 0_k \oplus v$ so that $T_{k+l}^{-1}(\mathfrak{u}_1 + \mathfrak{v}_1) = u \oplus v$. Thus 

\begin{eqnarray*}
\vert \mathfrak{u}_1 + \mathfrak{v}_1 \vert &=& T_{k+l}(\vert u \oplus v \vert_{k+l}) = T_{k+l}(\vert u \vert_k \oplus \vert v \vert_l) \\
&=& T_{k+l}(\vert u \vert_k \oplus 0_k) + T_{k+l}(0_k \oplus \vert v \vert_l) \\
&=& T_k(\vert u \vert_k) + T_{k+l}(\vert v_1 \vert_{k+l}) \\
&=& \vert \mathfrak{u}_1 \vert + \vert \mathfrak{v}_1 \vert. 
\end{eqnarray*}

Since $\mathfrak{r}\mathfrak{d}=\mathfrak{a}\mathfrak{a}^*\mathfrak{a},$ we get

\begin{eqnarray*}
\mathfrak{d}^* \mathfrak{u} \mathfrak{d} &=& \mathfrak{d}^*(\mathfrak{r}\mathfrak{u} \mathfrak{r})\mathfrak{d} \\
&=& (\mathfrak{d}^*\mathfrak{r})\mathfrak{u} (\mathfrak{r}\mathfrak{d})\\
&=& (\mathfrak{a}^*\mathfrak{a}\mathfrak{a}^*)\mathfrak{u}(\mathfrak{a}\mathfrak{a}^*\mathfrak{a})\\
&=& \mathfrak{a}^*(\mathfrak{a}\mathfrak{a}^*)\mathfrak{u} (\mathfrak{a}\mathfrak{a}^*)\mathfrak{a} \\
&=& \mathfrak{a}^* (\mathfrak{r}\mathfrak{u} \mathfrak{r}) \mathfrak{a} \\
&=& \mathfrak{a}^* \mathfrak{u} \mathfrak{a} \\
&=& \mathfrak{u}_1.
\end{eqnarray*} 

Similarly $\mathfrak{d}^* \mathfrak{v}\mathfrak{d} =\mathfrak{v}_1.$ Thus by Proposition \ref{19}, we get 
\begin{eqnarray*}
	\mathfrak{d}^*\vert \mathfrak{u}+\mathfrak{v} \vert \mathfrak{d} &=&\vert \mathfrak{d}^*(\mathfrak{u}+\mathfrak{v})\mathfrak{d} \vert=\vert \mathfrak{u}_1+\mathfrak{v}_1\vert \\
	&=&\vert \mathfrak{u}_1\vert+\vert \mathfrak{v}_1 \vert=\mathfrak{d}^*\vert \mathfrak{u} \vert\mathfrak{d} +\mathfrak{d}^*\vert \mathfrak{v} \vert\mathfrak{d} \\ 
	&=&\mathfrak{d}^*(\vert \mathfrak{u} \vert +\vert \mathfrak{v} \vert)\mathfrak{d}.
\end{eqnarray*}
Hence $\vert \mathfrak{u}+\mathfrak{v} \vert=\vert \mathfrak{u} \vert+\vert \mathfrak{v} \vert.$
\end{proof}

Next, we prove the converse of Theorem \ref{26}.

\begin{theorem}\label{z}
Let $(\mathfrak{V},\mathfrak{V}^{+},\vert \cdot \vert)$ be a non-degenerate absolutely ordered $\F$-bimodule. Let $(V,\lbrace M_n(V)^+ \rbrace)$ be the matrix ordered space whose matricial inductive limit is $(\mathfrak{V},\mathfrak{V}^+)$ as in Theorem \ref{w16}. For each $v \in M_n(V),n \in \mathbb{N}$ we define $\vert v \vert_n : = T_n^{-1}(\vert T_n(v) \vert)$ and consider the map ${\vert\cdot\vert}_{n}: M_{n}(V) \to M_{n}(V)^{+}$ given by $v \longmapsto \vert v\vert_{n}$. Then $(V,\lbrace{M_{n}(V)}^{+}\rbrace,\lbrace {\vert \cdot \vert}_{n}\rbrace)$ is an absolutely matrix ordered space with $(\mathfrak{V}, \mathfrak{V}^+, \vert \cdot \vert)$ as its matricial inductive limit.
\end{theorem}

\begin{proof}
Let $u \in M_m(V)$ and $v \in M_n(V)$. Put $v_1 = 0_m \oplus v$ and $a = \begin{bmatrix} 0 & I_m \\ I_n & 0\end{bmatrix}$. Then $a$ is unitary in $M_{m+n}$ with $a^* v_1 a = v \oplus 0_m$ so that $\sigma_{m+n}(a)$ is local unitary in $\mathfrak{F}$ with $\sigma_{m+n}(a)^* T_{m+n}(v_1) \sigma_{m+n}(a) = T_n(v)$. By Proposition \ref{19}, we get that $\sigma_{m+n}(a)^* \vert T_{m+n}(v_1) \vert \sigma_{m+n}(a) = \vert T_n(v) \vert $. Thus 

\begin{eqnarray*}
a^* \vert v_1 \vert_{m+n} a  &=& a^* T_{m+n}^{-1}(\vert T_{m+n}(v_1) \vert) a\\
&=& T_{m+n}^{-1}(\sigma_{m+n}(a)^* \vert T_{m+n}(v_1) \vert \sigma_{m+n}(a)) \\
&=& T_{m+n}^{-1}(\vert T_n(v) \vert) \\
&=& \vert v \vert_n \oplus 0_m
\end{eqnarray*}

so that $\vert v_1 \vert_{m+n} = 0_m \oplus \vert v \vert_n$. Now note that $T_m(u)$ and $T_{m+n}(v_1)$ are $\F$-independent in $\mathfrak{V}$. Thus $\vert T_m(u)+T_{m+n}(v_1) \vert = \vert T_m(u) \vert + \vert T_{m+n}(v_1) \vert$ so that 

\begin{eqnarray*}
\vert u \oplus v \vert_{m+n} &=& \vert (u \oplus 0_n) + (0_m \oplus v) \vert_{m+n} \\
&=& T_{m+n}^{-1}(\vert T_m(u)+T_{m+n}(v_1) \vert) \\
&=& T_{m+n}^{-1}(\vert T_m(u) \vert + \vert T_{m+n}(v_1) \vert) \\
&=& (\vert u \vert_m \oplus 0_n) + \vert v_1 \vert_{m+n} \\
&=& (\vert u \vert_m \oplus 0_n) + (0_m \oplus \vert v \vert_n) \\
&=& \vert u \vert_m \oplus \vert v \vert_n. 
\end{eqnarray*}
Now, it is routine to prove that $(V,\lbrace{M_{n}(V)}^{+}\rbrace,\lbrace {\vert \cdot \vert}_{n}\rbrace)$ is an absolutely matrix ordered space.
\end{proof}

\begin{definition}
Let $(\mathfrak{V},\mathfrak{V}^{+})$ be a non-degenerate $\F$-bimodule. An element $\mathfrak{e} \in \mathfrak{V}^{+}$ is said to be a local order unit for $(\mathfrak{V},\mathfrak{V}^{+})$, if it satisfies the following two conditions:
\end{definition}

\begin{enumerate}
\item[(i)] $\I_1 \mathfrak{e} \I_1 = \mathfrak{e}.$
\item[(ii)] For each $\mathfrak{v} \in \I_1 \mathfrak{V} \I_1,$ there exists $k > 0$ such that $${(k \mathfrak{e},k \mathfrak{e})_1}^{+}+sa_1(\mathfrak{v})\in \mathfrak{V}^{+}.$$
\end{enumerate}

\textbf{Note:-} Let $\mathfrak{V}$ be a non-degenerate $\F$-bimodule and let $\mathfrak{v} \in \mathfrak{V}$ with $\I_1 \mathfrak{v} \I_1 = \mathfrak{v}.$ For each $n \in \mathbb{N},$ we write $\mathfrak{v}^n = \displaystyle \sum_{i=1}^n \mathfrak{e}_{i,1} \mathfrak{v} \mathfrak{e}_{1,i}$.

\begin{proposition} \label{g}
Let $(\mathfrak{V},\mathfrak{V}^{+})$ be a non-degenerate $\F$-bimodule with a local order unit $\mathfrak{e}.$ Also let $(V, \lbrace M_n(V)^+\rbrace)$ be a matrix ordered space whose matricial inductive limit is $(\mathfrak{V},\mathfrak{V}^{+})$. Then
\begin{enumerate}
\item[(1)] For $\mathfrak{v} \in \I_n \mathfrak{V} \I_n$ and $k > 0$, we have ${(k \mathfrak{e}^n,k \mathfrak{e}^n)_n}^{+}+sa_n(\mathfrak{v})\in \mathfrak{V}^{+}$ if and only if ${(k \mathfrak{e}^n,k \mathfrak{e}^n)_n}^{+}-sa_n(\mathfrak{v})\in \mathfrak{V}^{+}.$ 
\item[(2)] $T^{-1}(\mathfrak{e})$ is an order unit for $V.$ 
\item[(3)] For each $\mathfrak{v} \in \I_n \mathfrak{V} \I_n,$ there exists $k > 0$ such that ${(k \mathfrak{e}^n,k \mathfrak{e}^n)_n}^{+}+sa_n(\mathfrak{v})\in \mathfrak{V}^{+}.$ 
\end{enumerate}
\end{proposition}

\begin{proof}
\begin{enumerate}
\item[(1)] Put $\mathfrak{a} = \I_n - \J_n^* \I_n \J_n$. Then $\mathfrak{a}$ is a local unitary with ${(k \mathfrak{e}^n,k \mathfrak{e}^n)_n}^{+}-sa_n(\mathfrak{v})=\mathfrak{a}^* ({(k \mathfrak{e}^n,k \mathfrak{e}^n)_n}^{+}+sa_n(\mathfrak{v}))\mathfrak{a}.$ Thus by Definition \ref{w46}(2), we have ${(k \mathfrak{e}^n,k \mathfrak{e}^n)_n}^{+}+sa_n(\mathfrak{v})\in \mathfrak{V}^{+}$ if and only if ${(k \mathfrak{e}^n,k \mathfrak{e}^n)_n}^{+}-sa_n(\mathfrak{v})\in \mathfrak{V}^{+}.$ 

\item[(2)] Since $\I_1 \mathfrak{e} \I_1 = \mathfrak{e}$, we get that $T^{-1}(\mathfrak{e}) \in V^+$. Next, let $v\in V$. Then $T(v) \in \I_1 \mathfrak{V} \I_1$. Now by definition and by (1), we have ${(k \mathfrak{e},k \mathfrak{e})_1}^{+}\pm sa_1(T(v)) \in \mathfrak{V}^+$. Note that $o({(k \mathfrak{e},k \mathfrak{e})_1}^{+}\pm sa_1(T(v))) \le 2$. Thus $\begin{bmatrix}k T^{-1}(\mathfrak{e}) & \pm v \\ \pm v^* & k T^{-1}(\mathfrak{e})\end{bmatrix} = T_{2}^{-1}({(k \mathfrak{e},k \mathfrak{e})_1}^{+}\pm sa_1(T(v)) \in M_2(V)^+$ so that $T^{-1}(\mathfrak{e})$ is an order unit for $V.$  

\item[(3)] Let $\mathfrak{v} \in \I_n \mathfrak{V} \I_n$ and put $e=T^{-1}(\mathfrak{e})$. Since $e$ is order unit for $V$ and $T_n^{-1}(\mathfrak{v}) \in M_n(V)$, we get that $\begin{bmatrix}k e^n & T_n^{-1}(\mathfrak{v}) \\ {T_n^{-1}(\mathfrak{v})}^* & ke^n\end{bmatrix} \in M_{2n}(V)^+$. Then ${(k\mathfrak{e}^n,k\mathfrak{e}^n)_n}^{+}+sa_n(\mathfrak{v}) = T_{2n}\left(\begin{bmatrix}k e^n & T_n^{-1}(\mathfrak{v}) \\ {T_n^{-1}(\mathfrak{v})}^* & ke^n\end{bmatrix}\right)  \in \mathfrak{V}^{+}.$   
\end{enumerate}
\end{proof}

\begin{definition}
A non-degenerate ordered $\F$-bimodule $(\mathfrak{V},\mathfrak{V}^{+})$ with a local order unit $\mathfrak{e}$ is said to be a local $\F$-order unit bimodule, if $\mathfrak{V}^+$ is proper and Archimedean. We denote it by $(\mathfrak{V},\mathfrak{V}^+,\mathfrak{e}).$
\end{definition}

\begin{remark}
Let $(V, \lbrace M_n(V)^+ \rbrace,e)$ be a matrix order unit space and let $(\mathfrak{V},\mathfrak{V}^+)$ be the matricial inductive limit of $(V, \lbrace M_n(V)^+ \rbrace)$ as in Theorem \ref{w16}. Then $(\mathfrak{V},\mathfrak{V}^{+},T(e))$ is a Local $\F$-order unit bimodule.
\end{remark}

\begin{proposition}\label{j}
Let $(\mathfrak{V},\mathfrak{V}^+,\mathfrak{e})$ be a local $\F$-order unit bimodule. For each $\mathfrak{v} \in \mathfrak{V},$ put $$\Vert \mathfrak{v} \Vert = \inf \lbrace k > 0: (k \mathfrak{e}^{o(\mathfrak{v})},k \mathfrak{e}^{o(\mathfrak{v})})_{o(\mathfrak{v})}^+ + sa_{o(\mathfrak{v})}(\mathfrak{v}) \in \mathfrak{V}^+\rbrace.$$ Then $(\mathfrak{V},\mathfrak{V}^+,\mathfrak{e})$ with $\Vert \cdot \Vert$ is a $\F$-bimodule normed space. Moreover, for each $\mathfrak{v} \in \mathfrak{V}$, we have $$\Vert \mathfrak{v} \Vert= \inf \lbrace k > 0: (k \mathfrak{e}^n,k \mathfrak{e}^n)_n^+ + sa_n(\mathfrak{v}) \in \mathfrak{V}^+~\textrm{for some}~n \in \mathbb{N}~\textrm{with}~\I_n \mathfrak{v} \I_n = \mathfrak{v} \rbrace.$$
\end{proposition}

\begin{proof}
Let $(V,\lbrace M_n(V)^+ \rbrace)$ be the matrix ordered space corresponding to $(\mathfrak{V},\mathfrak{V}^+)$ as in Theorem \ref{w16}. Put $e = T^{-1}(\mathfrak{e})$. By Remark \ref{h}(1) and (3) and by Proposition \ref{g}(2), we get that $(V,\lbrace M_n(V)^+\rbrace,e)$ is a matrix order unit space. Also recall that $(V,\lbrace M_n(V)^+\rbrace,e)$ with $\lbrace \Vert \cdot \Vert_n \rbrace$ is an $L^\infty$-matricially $\ast$-normed space, where each $\Vert \cdot \Vert_n$ is given in the following way:

$$\Vert v \Vert_n := \inf \left \lbrace k > 0: \begin{bmatrix} k e^n & v \\ v^* & k e^n \end{bmatrix} \in M_{2n}(V)^+ \right \rbrace$$ for all $v \in M_n(V)$.

Note that for any $\mathfrak{v} \in \mathfrak{V}$ and $k \in \mathbb{R}$, we have $(k \mathfrak{e}^n,k \mathfrak{e}^n)_n^+ + sa_n(\mathfrak{v}) \in \mathfrak{V}^+$ if and only if $\begin{bmatrix} ke^n & T_n^{-1}(\mathfrak{v}) \\ {T_n^{-1}(\mathfrak{v})}^* & ke^n \end{bmatrix} \in M_{2n}(V)^+$ for all $n \ge o(\mathfrak{v})$. Also note that for $n > o(\mathfrak{v})$, we have that 
$$\begin{bmatrix} ke^{o(\mathfrak{v})} & T_{o(\mathfrak{v})}^{-1}(\mathfrak{v}) \\ {T_{o(\mathfrak{v})}^{-1}(\mathfrak{v})}^* & ke^{o(\mathfrak{v})} \end{bmatrix} \in M_{2o(\mathfrak{v})}(V)^+$$ 

if and only if 

$$\begin{bmatrix} ke^n & T_n^{-1}(\mathfrak{v}) \\ {T_n^{-1}(\mathfrak{v})}^* & ke^n \end{bmatrix} \in M_{2n}(V)^+.$$ Thus $\Vert \mathfrak{v} \Vert = \Vert T_{o(\mathfrak{v})}^{-1} (\mathfrak{v}) \Vert_{o(\mathfrak{v})}$ and 
$(k \mathfrak{e}^{o(\mathfrak{v})},k \mathfrak{e}^{o(\mathfrak{v})})_{o(\mathfrak{v})}^+ + sa_{o(\mathfrak{v})}(\mathfrak{v}) \in \mathfrak{V}^+$ if and only if 
$(k \mathfrak{e}^n, k \mathfrak{e}^n)_n^+ + sa_n (\mathfrak{v}) \in \mathfrak{V}^+$. Since $\Vert v \oplus 0_m \Vert_{n+m} = \Vert v \Vert_n$ for all $v \in M_n(V)$ and $n,m \in \mathbb{N}$, we get that $\Vert \cdot \Vert$ determines a norm on $\mathfrak{V}$ and 
$\Vert \mathfrak{v} \Vert= \inf \lbrace k > 0: (k \mathfrak{e}^n,k \mathfrak{e}^n)_n^+ + sa_n(\mathfrak{v}) \in \mathfrak{V}^+~\textrm{for some}~n \in \mathbb{N}~ \textrm{with} ~\I_n \mathfrak{v} \I_n = \mathfrak{v} \rbrace.$

Finally, let $\mathfrak{a},\mathfrak{b} \in \mathfrak{F}$ and $\mathfrak{v} \in \mathfrak{V}$. Put $n = \textrm{max} \lbrace o(\mathfrak{a}), o(\mathfrak{b}), o(\mathfrak{v}) \rbrace$. Then 
\begin{eqnarray*}
\Vert \mathfrak{a} \mathfrak{v} \mathfrak{b}\Vert & = & \Vert T_{o(\mathfrak{a} \mathfrak{v} \mathfrak{b})}^{-1}(\mathfrak{a} \mathfrak{v} \mathfrak{b}) \Vert_{o(\mathfrak{a} \mathfrak{v} \mathfrak{b})} = \Vert T_n^{-1}(\mathfrak{a} \mathfrak{v} \mathfrak{b}) \Vert_n  \\
&=& \Vert \sigma_n^{-1}(\mathfrak{a}) T_n^{-1}(\mathfrak{v}) \sigma_n^{-1}(\mathfrak{b}) \Vert_n \\
& \le & \Vert \sigma_n^{-1}(\mathfrak{a}) \Vert \Vert T_n^{-1}(\mathfrak{v}) \Vert_n \Vert\sigma_n^{-1}(\mathfrak{b}) \Vert \\
&=& \Vert \mathfrak{a} \Vert \Vert \mathfrak{v} \Vert \Vert \mathfrak{b} \Vert.
\end{eqnarray*}
so that $\Vert \cdot \Vert$ is a $\mathfrak{F}$-bimodule norm on $\mathfrak{V}$. Hence $(\mathfrak{V},\mathfrak{V}^+,\mathfrak{e})$ with $\Vert \cdot \Vert$ is a $\F$-bimodule normed space.
\end{proof}

\begin{definition}
Let $(\mathfrak{V},\mathfrak{V}^+)$ be a non-degenerate ordered $\F$-bimodule and let $\Vert \cdot \Vert$ be a $\mathfrak{F}$-bimodule norm on $\mathfrak{V}$ . Also let $\mathfrak{u}, \mathfrak{v} \in \mathfrak{V}^+.$ We write $\mathfrak{u} \perp_\infty \mathfrak{v}$, if $$\Vert k_1 \mathfrak{u} + k_2 \mathfrak{v} \Vert = \max \lbrace \Vert k_1 \mathfrak{u} \Vert, \Vert k_2 \mathfrak{v} \Vert \rbrace$$ for all $k_1, k_2 \in \mathbb{R}$ and $\mathfrak{u} \perp_\infty^a \mathfrak{v}$, if $\mathfrak{u}_1 \perp_\infty \mathfrak{v}_1$ for all $\mathfrak{u}_1,\mathfrak{v}_1 \in \mathfrak{V}^+$ with $\mathfrak{u}_1 \le \mathfrak{u}$ and $\mathfrak{v}_1 \le \mathfrak{v}.$
\end{definition}

\begin{proposition}
Let $(\mathfrak{V},\mathfrak{V}^+,\mathfrak{e})$ be a local $\F$-order unit bimodule and let $\mathfrak{u}, \mathfrak{v} \in \mathfrak{V}^+$ such that $\Vert \mathfrak{u}\Vert =1= \Vert \mathfrak{v} \Vert.$ Then $\mathfrak{u} \perp_\infty \mathfrak{v}$ if and only if $\Vert \mathfrak{u} + \mathfrak{v} \Vert = 1.$ 
\end{proposition}

\begin{proof}
Let $(V,\lbrace M_n(V)^+\rbrace)$ be the matrix ordered space corresponding to $(\mathfrak{V},\mathfrak{V}^+)$ as in Theorem \ref{w16}. Put $e=T^{-1}(\mathfrak{e})$. Then $(V, \lbrace M_n(V)^+ \rbrace, e)$ is a matrix order unit space.
By proof of Proposition \ref{j}, we get that $\Vert \mathfrak{v} \Vert = \Vert T_n^{-1}(\mathfrak{v}) \Vert_n$ for any $\mathfrak{v} \in \mathfrak{V}$ and $n \ge o(\mathfrak{v})$. Also note that $\mathfrak{v} \in \mathfrak{V}^+$ if and only if $T_n^{-1}(\mathfrak{v}) \in M_n(V)^+$ for all $n \ge o(\mathfrak{v})$. 

Next let $\mathfrak{u}, \mathfrak{v} \in \mathfrak{V}^+$ with $\Vert \mathfrak{u}\Vert =1= \Vert \mathfrak{v} \Vert.$ Without loss of generality, we may assume that $o(\mathfrak{u}) \ge o(\mathfrak{v})$. Then $\mathfrak{u} \perp_\infty \mathfrak{v}$ in $\mathfrak{V}^+$ if and only if $T_{o(\mathfrak{u})}^{-1}(\mathfrak{u}) \perp_\infty T_{o(\mathfrak{u})}^{-1}(\mathfrak{v})$ in $M_{2o(\mathfrak{u})}(V)$. By \cite[Theorem 3.3]{K14}, we have $T_{o(\mathfrak{u})}^{-1}(\mathfrak{u}) \perp_\infty T_{o(\mathfrak{u})}^{-1}(\mathfrak{v})$ if and only if $\Vert T_{o(\mathfrak{u})}^{-1}(\mathfrak{u}) + T_{o(\mathfrak{u})}^{-1}(\mathfrak{v})\Vert_{o(\mathfrak{u})} = 1$. Thus $\mathfrak{u} \perp_\infty \mathfrak{v}$ if and only if $\Vert \mathfrak{u} + \mathfrak{v} \Vert = 1$.
\end{proof}

\begin{definition}
Let $(\mathfrak{V},\mathfrak{V}^+,\mathfrak{e})$ is a Local $\F$-order unit bimodule such that   
\begin{enumerate}
\item[(a)]$(\mathfrak{V},\mathfrak{V}^+,\vert \cdot \vert)$ is a non-degenerate absolutely ordered $\F$- bimodule. and
\item[(b)]$\perp = \perp_\infty^a$ on $\mathfrak{V}^+.$
\end{enumerate}
Then $(\mathfrak{V},\mathfrak{V}^+,\mathfrak{e},\vert \cdot \vert)$ is said to be non-degenerate absolute order unital $\F$-bimodule.
\end{definition}

Next two results follow from Theorems \ref{26} and \ref{z} respectively.

\begin{corollary}\label{27}
Let $(V,\lbrace M_n(V)^+ \rbrace, e , \lbrace \vert \cdot \vert_n\rbrace)$ be an absolute matrix order unit space. Let $(\mathfrak{V},\mathfrak{V}^+,\vert \cdot \vert)$ be non-degenerate absolutely ordered $\F$-bimodule corresponding to $(V,\lbrace M_n(V)^+ \rbrace,\lbrace \vert \cdot \vert_n\rbrace)$ as in Theorem \ref{26}. Then $(\mathfrak{V},\mathfrak{V}^+,T(e),\vert \cdot \vert)$ is a non-degenerate absolute order unital $\F$-bimodule. 
\end{corollary}

\begin{corollary}\label{l}
Let $(\mathfrak{V},\mathfrak{V}^+,\mathfrak{e},\vert \cdot \vert)$ be a non-degenerate absolute order unital $\F$-bimodule. Let $(V,\lbrace M_n(V)^+ \rbrace,\lbrace \vert \cdot \vert_n\rbrace)$ be absolute matrix ordered space corresponding to $(\mathfrak{V},\mathfrak{V}^+,\vert \cdot \vert)$ as in Theorem \ref{z}. Then $(V,\lbrace M_n(V)^+ \rbrace, T^{-1}(\mathfrak{e}) , \lbrace \vert \cdot \vert_n\rbrace)$ is an absolute matrix order unit space.
\end{corollary}

\section{$K_0$-group corresponding to an absolute matrix order unit space}

In this section, we prove the existence of $K_0$-group of an absolute matrix order unit space with its explicit form. 

\begin{definition}\cite[Definition 3.1]{PI19}
Let $(V,e)$ be an absolute matrix order unit space and let $v \in M_n(V)$ for some $n\in \mathbb{N}.$ Then $v$ is said to be \emph{order projection}, if $v^* = v$ and $\vert 2 v - e^n \vert_n = e^n.$ 

An equivalent definition may be extended to non-degenerate absolute order unital $\F$-bimodules. Let $(\mathfrak{V},\mathfrak{V}^+,\mathfrak{e},\vert \cdot \vert)$ be a non-degenerate absolute order unital $\F$-bimodule. Then for $\mathfrak{p} \in \mathfrak{V}^+$ with $\Vert \mathfrak{p} \Vert \le 1$, we say that $\mathfrak{p}$ is an order projection, if $\mathfrak{p}^* = \mathfrak{p}$ and $\vert \mathfrak{e}^n - 2 \mathfrak{p} \vert = \mathfrak{e}^n$ whenever $n \ge o(\mathfrak{p})$.

Following the constructions of the previous section, we may identify the direct limit of an absolute matrix order unit space $V$ with $M_\infty(V) = \displaystyle \bigcup_{n=1}^\infty M_n(V)$. Further, under this identification, the corresponding set of projections may be identified with $\mathcal{OP}_\infty(V) = \displaystyle \bigcup_{n=1}^\infty \mathcal{OP}_n(V)$.

Now let $v \in M_{m,n}(V)$ for some $m, n \in \mathbb{N}.$ Then $v$ is said to be \emph{partial isometry}, if $\vert v \vert_{m,n}$ and $\vert v^* \vert_{n,m}$ are order projections.

The set of all partial isometries in $M_{m,n}(V)$ will be denoted by $\mathcal{PI}_{m,n}(V)$ and the set of all order projections in $M_n(V)$ will be denoted by $\mathcal{OP}_n(V)$. For $m = n$, we write $\mathcal{PI}_{m,n}(V) = \mathcal{PI}_n(V)$. For $n = 1$, we shall write $\mathcal{PI}(V)$ for $\mathcal{PI}_1(V)$ and $\mathcal{OP}(V)$ for $\mathcal{OP}_1(V)$. 
\end{definition}

\begin{definition}\cite[Definition 4.1]{PI19}
Let $V$ be an absolute matrix order unit space. Consider the set of order projections in the corresponding direct limit $M_{\infty}(V)$: 
$$\mathcal{OP}_{\infty}(V)= \lbrace p: p \in \mathcal{OP}_n(V) ~ \textrm{for some} ~ n \in \mathbb{N} \rbrace.$$
Given $p\in \mathcal{OP}_m(V)$ and $q\in \mathcal{OP}_n(V)$, we say that $p$ is partial isometric equivalent to $q$, (we write, $p\sim q$), if there exists $v\in \mathcal{PI}_{m,n}(V)$ such that $p=\vert v^*\vert_{n,m}$ and $q=\vert v\vert_{m,n}$. 
\end{definition}
We recall the following properties of this relation.
\begin{proposition} \cite[Propositions 4.1 and 4.2]{PI19} \label{15}
	Let $V$ be an absolute matrix order unit space and let $p,q,r,p',q'\in \mathcal{OP}_\infty(V)$. Then 
	\begin{enumerate}
		\item If $m, n \in \mathbb{N}$ and let $p\in\mathcal{OP}_m(V)$, then $p\sim p\oplus 0_n$ and $p \sim 0_n \oplus p$; 
		\item If $p \sim q$ and $p' \sim q'$ with $p \perp p'$ and $q \perp q'$, then $p + p' \sim q + q'$;
		\item If $p\sim p'$ and $q\sim q',$ then $p\oplus q\sim p'\oplus q';$
		\item $p\oplus q\sim q\oplus p;$
		\item If $p,q\in \mathcal{OP}_n(V)$ for some $n\in \mathbb{N}$ such that $p\perp q,$ then $p+q\sim p\oplus q;$
		\item $(p\oplus q)\oplus r=p\oplus (q\oplus r);$ 
		\item $\sim$ is an equivalence relation in $\mathcal{OP}_\infty(V)$, if the following condition holds:
		
		\textbf{(T):} If $u \in \mathcal{PI}_{m,n}(V)$ and $v \in \mathcal{PI}_{l,n}(V)$ with $\vert u \vert_{m,n} = \vert v \vert_{l,n}$, then there exists a $w \in \mathcal{PI}_{m,l}(V)$ such that $\vert w^* \vert_{l,m} = \vert u^* \vert_{n,m}$ and $\vert w \vert_{m,l} = \vert v^* \vert_{n,l}$.
	\end{enumerate} 
\end{proposition} 

 We extend $\sim$ to the following relation on $\mathcal{OP}_\infty(V)$:

\begin{definition}
Let $V$ be an absolute matrix order unit space. For $p,q\in \mathcal{OP}_\infty(V),$ we say that $p\approx q$, if there exists $r\in \mathcal{OP}_\infty(V)$ such that $p\oplus r\sim q\oplus r.$
\end{definition}

\begin{proposition}\label{16}
Let $V$ be an absolute matrix order unit space and let $p,q \in \mathcal{OP}_\infty(V)$. 
\begin{enumerate}
\item[(1)] $p \sim q$ implies $p \approx q$.  
\item[(2)] If $V$ satisfies condition (T), then 
\begin{enumerate}
 \item[(i)]$\approx$ is an equivalence relation. 
\item[(ii)] $p\approx q$ if and only if $p \oplus e^m \sim q \oplus e^m$ for some $m \in \mathbb{N}.$ 
\end{enumerate}
\end{enumerate}
\end{proposition}

\begin{proof}
Let $p \sim q$ and also let $r \in \mathcal{OP}_\infty(V)$. Since $r \sim r$, by Proposition \ref{15}(3), we get that $p\oplus r \sim q \oplus r$. Thus (1) follows.

Next, we assume that $V$ satisfies condition (T). By Proposition \ref{15} and by (1), it follows that $\approx$ is reflexive and symmetric. Next, we check transitivity of $\approx$. Let $p \approx q$ and also let $r \in \mathcal{OP}_\infty(V)$ such that $q \approx r$. Then $p \oplus s \sim q \oplus s$ and $q \oplus t \sim r \oplus t$ for some $s,t \in \mathcal{OP}_\infty(V)$. Thus by Propositions \ref{15}, we get that $p \oplus (s \oplus t) \sim r \oplus (s \oplus t)$ so that $p \approx r$. 

Let $p \approx q$. Then there exists $r \in \mathcal{OP}_m(V), m \in \mathbb{N}$ such that $p \oplus r \sim q \oplus r$. Since $r \perp (e^m-r)$, by Proposition \ref{15}(5), we get that $e^m \sim r \oplus (e^m - r)$. Then again applying Proposition \ref{15}(3), we conclude that $p \oplus e^m \sim p \oplus (r \oplus (e^m - r)),~(p \oplus r) \oplus (e^m-r) 
\sim (q \oplus r) \oplus (e^m-r),~q \oplus (r \oplus (e^m-r)) \sim q \oplus e^m.$ Thus by Proposition \ref{15}, we have that $p \oplus e^m \sim q \oplus e^m$.
\end{proof}

\begin{proposition}\label{54}
	Let $V$ be an absolute matrix order unit space and let $p,q,r,p',q'\in \mathcal{OP}_\infty(V)$. 
	\begin{enumerate}
		\item If $m, n \in \mathbb{N}$ and let $p\in\mathcal{OP}_m(V)$, then $p \approx p\oplus 0_n$ and $p \approx 0_n \oplus p$; 
		\item If $V$ satisfies (T), then
		\begin{enumerate}
		\item[(i)] If $p \approx q$ and $p' \approx q'$ with $p \perp p'$ and $q \perp q'$, then $p + p' \approx q + q'$;
		\item[(ii)] If $p\approx p'$ and $q\approx q',$ then $p\oplus q\approx p'\oplus q';$
		\end{enumerate}
		\item $p\oplus q\approx q\oplus p;$
		\item If $p\perp q,$ then $p+q\approx p\oplus q;$
		\item $(p\oplus q)\oplus r=p\oplus (q\oplus r).$
	\end{enumerate} 
\end{proposition}

\begin{proof}
(1), (3), (4) and (5) are immediate by Proposition \ref{15}(1), (4), (5) and (6) respectively with Proposition \ref{16}(1).

Next, we prove (2). Assume that $V$ satisfies (T).
Let $p \approx q,~p' \approx q',~p \perp p'$ and $q \perp q'$. Then $p \oplus r \sim q \oplus r$ and $p' \oplus s \sim q' \oplus s$ for some $r,s \in \mathcal{OP}_\infty(V)$. Without loss of generality, we can assume that $r \perp s$ so that $p \oplus r \perp p' \oplus s$ and $q \oplus r \perp q' \oplus s$. By Proposition \ref{15}(2), we get that $(p+p') \oplus (r+s)=(p \oplus r) + (p' \oplus s) \sim (q \oplus r)+(q' \oplus s)=(q+q') \oplus (r+s)$ so that $p+p' \approx q+q'$.

Now let $p \approx p'$ and $q \approx q'$. Then there exist $r,s \in \mathcal{OP}_\infty(V)$ such that $p \oplus r \sim p' \oplus r$ and $q \oplus s \sim q' \oplus s$. Now result follows from Propositions \ref{15} and \ref{14}.
\end{proof}

\begin{proposition}\label{w}
Let $V$ be an absolute matrix order unit space satisfying (T). For each $p,q\in \mathcal{OP}_\infty(V),$ let $[p] = \lbrace r\in \mathcal{OP}_\infty(V): r\approx p \rbrace$ and put $[p]+ [q] = [p\oplus q].$ Then $+$ is a binary operation in the family of equivalence classes $(\mathcal{OP}_\infty(V)\big/\approx, +)$. Also
\begin{enumerate}
\item[(1)] $[p] + [0] = [p]$ for all $p\in \mathcal{OP}_\infty(V);$
\item[(2)] $[p]+ [q] = [q]+ [p]$ for all $p,~q\in \mathcal{OP}_\infty(V);$
\item[(3)] $[p]+ [r] = [q]+ [r]$ for $p,~q,~r\in \mathcal{OP}_\infty(V),$ then $[p] = [q].$ 
\end{enumerate}
Thus $(\mathcal{OP}_\infty(V)\big/\approx, +)$ is an abelian semigroup satisfying the cancellation law.
\end{proposition}
\begin{proof}
By Proposition \ref{54}(2)(ii), it follows that $+$ is well-defined in $\mathcal{OP}_\infty(V)\big/\approx.$ Note that (1) and (2) immediately follow from \ref{54}(1) and (3) respectively. Next, we prove (3).

Let $p,q,r \in \mathcal{OP}_\infty(V)$ such that $[p]+ [r] = [q]+ [r]$. Then $p\oplus r \approx q\oplus r$  so that $p\oplus (r \oplus s) \sim q\oplus (r \oplus s)$ for some $s\in \mathcal{OP}_\infty(V)$. Thus $p \approx q$ so that $[p] = [q].$

\end{proof}
Now we construct the Grotheindick group from the abelian semigroup developed in the previous result.

\begin{theorem}\label{w38}
Let $(V,e)$ be an absolute matrix order unit space satisfying (T) and consider $\mathcal{OP}_\infty(V) \times \mathcal{OP}_\infty(V)$. For all $p_1,p_2,q_1,q_2 \in \mathcal{OP}_\infty(V),$ we define $(p_1,q_1) \equiv (p_2,q_2)$ if and only if $p_1\oplus q_2 \approx p_2\oplus q_1.$ Then $\equiv$ is an equivalence relation on $\mathcal{OP}_\infty(V) \times \mathcal{OP}_\infty(V).$ Consider $$K_0(V)=\lbrace [(p_,q)]: p,q\in \mathcal{OP}_\infty\rbrace$$ where $[(p,q)]$ is the equivalence class of $(p,q)$ in $(\mathcal{OP}_\infty(V) \times \mathcal{OP}_\infty(V),\equiv).$ For all $p_1,p_2,q_1,q_2 \in \mathcal{OP}_\infty(V),$ we write $$[(p_1,q_1)]+ [(p_2,q_2)] = [(p_1\oplus p_2, q_1\oplus q_2)].$$ Then $(K_0(V),+)$ is an abelian group.
\end{theorem}

\begin{proof}
It follows from Proposition \ref{16}(2)(i) that the relation $\equiv$ on $\mathcal{OP}_\infty(V) \times \mathcal{OP}_\infty(V)$ is reflexive and symmetric.  Let $[(p_1,q_1)] \equiv [(p_2,q_2)]$ and $[(p_2,q_2)] \equiv [(p_3,q_3)]$ for some $p_1,p_2,p_3,q_1,q_2,q_3 \in \mathcal{OP}_\infty(V)$. Then $p_1 \oplus q_2 \approx p_2 \oplus q_1$ and $p_2 \oplus q_3 \approx p_3 \oplus q_2$. By Proposition \ref{54}(2)(ii),(3) and (5), we get that $(p_1 \oplus q_3) \oplus (p_2 \oplus q_2) \approx (p_3 \oplus q_1) \oplus (p_2 \oplus q_2)$ so that $[p_1 \oplus q_3] + [p_2 \oplus q_2] = [p_3 \oplus q_1] + [p_2 \oplus q_2]$. By Proposition \ref{w}(3), we conclude that $[p_1 \oplus q_3] = [p_3 \oplus q_1]$ so that $p_1 \oplus q_3 \approx p_3 \oplus q_1$. Thus $[(p_1,q_1)] \equiv [(p_3,q_3)]$ so that $\equiv$ is transitive. Hence $\equiv$ is an equivalence relation on $\mathcal{OP}_\infty(V) \times \mathcal{OP}_\infty(V).$ 

Now, we show that $+$ is well-defined on $K_0(V).$ Let $(p_1,q_1) \equiv (p_1',q_1')$ and $(p_2,q_2) \equiv (p_2',q_2')$ in $\mathcal{OP}_\infty(V) \times \mathcal{OP}_\infty(V).$ Then $p_1\oplus q_1' \approx p_1' \oplus q_1$ and $p_2\oplus q_2' \approx p_2' \oplus q_2$. By Proposition \ref{54}(2)(ii),(3) and (5), we get that $(p_1\oplus p_2)\oplus (q_1^,\oplus q_2^,)\approx (p_1^,\oplus p_2^,)\oplus (q_1\oplus q_2).$ Thus $[(p_1, q_1)]+[(p_2, q_2)] = [(p_1', q_1')]+[(p_2', q_2')]$ so that $+$ is well-defined. 

By Proposition \ref{54}(2)(ii) and (3), we have 
$$(p_1\oplus p_2)\oplus (q_2\oplus q_1) \approx (p_2\oplus p_1)\oplus (q_1\oplus q_2)$$ for all $p_1,q_1,p_2,q_2\in \mathcal{OP}_\infty(V)$ so that $+$ is commutative in $K_0(V)$.

Let $p,q\in \mathcal{OP}_\infty(V).$ By Proposition \ref{54}(1) and \ref{54}(2)(ii), we have $(p\oplus 0)\oplus (q \oplus 0) \approx p\oplus q.$ Thus $[(0,0)]$ is an identity element in $K_0(V).$  

Associativity of $+$ on $K_0(V)$ follows from \ref{54}(5).

For any $p~\textit{and}~q\in \mathcal{OP}_\infty(V),$ we have that  $(p\oplus q, q\oplus p) \equiv (0,0)$ so that $[(p\oplus q, q\oplus p)] \equiv [(0,0)]$. Hence $[(q,p)]$ is inverse element of $[(p,q)]$ in $K_0(V).$ 
\end{proof}
\begin{remark}
	The abelian group $K_0(V)$ is called the $K_0$-group of the absolute matrix order unit space $V$.
\end{remark}

\subsection{Order structure in $K_0$}

Let $(G,+)$ be an abelian group and let $G^+$ be a non-empty subset of $G$ such that $0 \in G^+$ and $G^+ + G^+ \subset G^+$. Then $G^+$ is said to be a group cone in $G$. In this case, $(G,G^+)$ is said to be an ordered abelian group. In this case, given $g_1,g_2 \in G$, we define $g_1 \le g_2$ if and only if $g_2 - g_1 \in G^+$. Then $\le$ is reflexive, transitive and translation invariant. Further, we have $G^+ = \lbrace g \in G : g \ge 0 \rbrace$. Conversely, if $\le$ is a reflexive, transitive and translation invariant relation in an additive abelian group $G$, then $G^+ := \lbrace g \in G : g \ge 0 \rbrace$ is a group cone in $G$ so that $(G,G^+)$ is an ordered abelian group.

Let $(G,G^+)$ be an ordered abelian group. Then   
\begin{enumerate}
	\item[(1)] $G^+$ is said to be proper, if $G^+ \cap - G^+ = \lbrace 0 \rbrace$.
	\item[(2)] $G^+$ is said to be generating, if $G = G^+ - G^+$.
\end{enumerate}
For the further discussion, we shall assume that $(G,G^+)$ is proper and generating.

Let $(G,G^+)$ be an ordered group. Then $g \in G^+$ is said to be an order unit for $G$, if given any $h \in G$ there exists $n \in \mathbb{N}$ such that $-ng \le h \le ng$.

An ordered abelian group $(G,G^+)$ with an order unit $g$, is called an ordered abelian group with a distinguished order unit. For details please refer to \cite{KRG86, RLL00}.

In this subsection, we show that given an absolute matrix order unit space $V$, $K_0(V)$ is an ordered abelian group with a distinguished order unit under the assumption that all order projections in $V$ are finite. 
\begin{theorem}\label{23}
	Let $V$ be an absolute matrix order unit space satisfying (T). Put $K_0(V)^+ = \lbrace [(p,0)]: p\in \mathcal{OP}_\infty(V)\rbrace.$ Then 
	\begin{enumerate}
		\item[(a)]$K_0(V)^+$ is a group cone in $K_0(V)$. 
		\item[(b)]If $e^n$ is finite for all $n \in \mathbb{N}$, then $K_0(V)^+$ is proper. 
		\item[(c)] $K_0(V)^+$ is generating. 
	\end{enumerate} 
	In other words, if $e^n$ is finite for all $n \in \mathbb{N}$, then $(K_0(V),K_0(V)^+)$ is an ordered abelian group.
\end{theorem}

\begin{proof}
	It is routine to verify (a) and (c). Next, we prove (b). Assume that $e^n$ is finite for all $n \in \mathbb{N}$. Let $g\in K_0(V)^+ \cap -K_0(V)^+.$ There exist $p \in \mathcal{OP}_m(V)$ and $q\in \mathcal{OP}_n(V)$ such that $g=[(p,0)] = [(0,q)]$. Then $(p,0) \equiv (0,q)$ so that $p \oplus q \approx 0_m \oplus 0_n$. Thus $p\oplus q \oplus r \sim 0_m \oplus 0_n \oplus r$ for some $r\in \mathcal{OP}_l(V).$ Since $p\oplus q \oplus r \sim 0_m \oplus 0_n \oplus r$ and $p \oplus q \oplus r \ge 0_m \oplus 0_n \oplus r,$ by \cite[Corollary 5.3]{PI19}, we get that $p \oplus q \oplus r = 0_m \oplus 0_n \oplus r$. Then $p \oplus q \oplus 0_l = 0_{m+n+l}$ so that $p = 0_m$ and $q=0_n.$ Thus $g=0.$ 
\end{proof}
\begin{corollary}\label{w13}
	Let $(V,e)$ be an absolute matrix order unit space satisfying (T) and let $e^n$ be finite for all $n \in \mathbb{N}$. Then $(K_0(V),K_0(V)^+)$ is an ordered abelian group with distinguished order unit $[(e,0)].$ In other words, for each $g\in K_0(V),$ there exists $n\in \mathbb{N}$ such that $-n[(e,0)]\le g\le n[(e,0)].$    
\end{corollary}

\begin{proof}
	By Theorem \ref{23}, $(K_0(V),K_0(V)^+)$ is an ordered abelian group. We show that $[(e,0)]$ is an order unit. First, let $r\in \mathcal{OP}_m(V)$ for some $m \in \mathbb{N}.$ By \cite[Proposition 4.5(5)]{PI19}, we have 
	\begin{eqnarray*}
		[(r,0)] & \le & [(r,0)]+[(e^m-r,0)] \\
		&=& [(r\oplus (e^m-r),0)] \\
		&=& [(r+(e^m-r),0)]\\
		&=& [(e^m,0)]\\
		&=& m[(e,0)].
	\end{eqnarray*}
	
	Let $g\in  K_0(V).$ Then by Theorem \ref{23}(c), we have $g=[(p,0)]-[(q,0)]$ for some $p,q\in \mathcal{OP}_\infty(V).$ Without any loss of generality, we may assume that $p,q\in \mathcal{OP}_n(V)$ for some $n\in \mathbb{N}.$ Since $-[(q,0)] \le g \le [(p,0)],$ we get that $-n[(e,0)] \le g \le n[(e,0)].$ Hence $[(e,0)]$ is an order unit for $K_0(V).$
\end{proof}

\section{Functoriality of $K_0$}

Let $V$ be an absolute matrix order unit space satisfying (T). Then $p \longmapsto [(p,0)]$ defines a map $\chi_V:\mathcal{OP}_\infty(V) \to K_0(V)$. If $V$ and $W$ are complex vector spaces and if $\phi: V \to W$ be a linear map, we denote the corresponding $\mathcal{F}$-bimodule map from $M_\infty(V)$ to $M_\infty(W)$ again by $\phi$. In this sense, $\phi_{\mid_{M_n(V)}} = \phi_n$ for all $n \in \mathbb{N}$. In the next result, we describe the functorial nature of $K_0$. For this, first we need to recall the following definition:

\begin{definition}\cite[Definition 4.5]{K19}
Let $V$ and $W$ be absolute matrix order unit spaces and let $\phi: V \to W$ be a $\ast$-linear map. We say that $\phi$ is  \emph{completely $\vert\cdot\vert$-preserving}, if $\phi_n: M_n(V) \to M_n(W)$ is an $\vert\cdot\vert$-preserving map for each $n \in \mathbb{N}$.
\end{definition}

\begin{theorem}\label{51}
Let $V$ and $W$ be absolute matrix order unit spaces satisfying (T) and let $\phi: V \to W$ be a completely $\vert \cdot \vert$-preserving map such that $\phi(e_V) \in \mathcal{OP}(W)$. Then there exists a unique group homomorphism $K_0(\phi): K_0(V)\to K_0(W)$ such that the following diagram commutes: 
\end{theorem}

$$\begin{tikzcd}
\mathcal{OP}_\infty(V)\arrow{r}{\phi}\arrow[swap]{d}{\chi_V}
& \mathcal{OP}_\infty(W)\arrow{d}{\chi_W} \\
K_0(V)\arrow[swap]{r}{K_0(\phi)}
& K_0(W)
\end{tikzcd}$$

\begin{proof}
Since $\phi(e) \in \mathcal{OP}(W)$, by \cite[Theorem 3.7]{K19} and \cite[Proposition 3.3]{PI19}, we have that $\phi(\mathcal{OP}_\infty(V)) \subset \mathcal{OP}_\infty(W)$. Let $p,q \in \mathcal{OP}_{\infty}(V)$ such that $p \sim q$. Without loss of generality, we may assume that $p \in \mathcal{OP}_n(V)$ and $q \in \mathcal{OP}_m(V)$ with $n \ge m.$ There exists $v \in M_{m,n}(V)$ such that $p = \vert v \vert_{m,n}$ and $q = \vert v^* \vert_{n,m}.$ Put $r = n - m$ and $w = \begin{bmatrix} v\\0\end{bmatrix}\in M_n(V).$ By Proposition \ref{151}(4) and (5), it follows that $p = \vert w \vert_n$ and $q \oplus 0_r = \vert w^* \vert_n$ so that $p \sim q\oplus 0_r$ in $\mathcal{OP}_n(V).$ As $\phi_n: M_n(V) \to M_n(W)$ is an $\vert \cdot \vert$-preserving map, we get that $\phi_n(p) = \vert \phi_n(w) \vert_n$ and $\phi_m(q)\oplus 0_r = \vert \phi_n(w)^* \vert_n.$ Thus $\phi_n(p) \sim \phi_m(q)\oplus 0_r$. By \cite[Proposition 4.4]{K19} and \cite[Proposition 4.5(1)]{K19}, we have that $\phi_n(p) \sim \phi_m(q)$ so that $\phi(p) \sim \phi(q)$.

Next, put $K_0(\phi)([(p,q)])=[(\phi(p),\phi(q))]$ for each $[(p,q)]\in K_0(V)$. We show that $K_0(\phi)$ is well-defined. Let $[(p_1,q_1)] = [(p_2,q_2)]$ for some $p_1,p_2,q_1,q_2 \in \mathcal{OP}_\infty(V).$ Then there exists $r\in \mathcal{OP}_\infty(V)$ such that $p_1\oplus q_2 \oplus r \sim p_2 \oplus q_1 \oplus r.$ Thus $\phi(p_1)\oplus \phi(q_2) \oplus \phi(r) \sim \phi(p_2) \oplus \phi(q_1) \oplus \phi(r)$ so that $[(\phi(p_1),\phi(q_1))] = [(\phi(p_2),\phi(q_2))].$ Hence $K_0(\phi)$ is well-defined. For all $[(p_1,q_1)],[(p_2,q_2)] \in K_0(V)$, we have that 

\begin{eqnarray*}
K_0(\phi)([(p_1,q_1)]+[(p_2,q_2)]) &=& K_0(\phi)([(p_1 \oplus p_2,q_1 \oplus q_2)]) \\
&=& [(\phi(p_1 \oplus p_2),\phi(q_1 \oplus q_2))] \\
&=& [(\phi(p_1) \oplus \phi(p_2),\phi(q_1) \oplus \phi(q_2))] \\
&=& [(\phi(p_1),\phi(q_1))] + [(\phi(p_2),\phi(q_2))] \\
&=& K_0(\phi)([(p_1,q_1)]) + K_0(\phi)([(p_2,q_2)])
\end{eqnarray*}

so that $K_0(\phi)$ is a group homomorphism. By construction $K_0$ satisfies the diagram.

\textbf{Uniqueness of $K_0(\phi)$:-} Let $\mathcal{H}: K_0(V) \to K_0(W)$ be a group homomorphism satisfying the same diagram. Then $K_0(\phi)(\chi_V(p))=\chi_W(\phi(p))=\mathcal{H}(\chi_V(p))$ for all $p \in \mathcal{OP}_\infty(V)$. Thus we get that

\begin{eqnarray*}
K_0(\phi)([(p,q)]) &=& K_0(\phi)([(p,0)]-[(q,0)]) \\
&=& K_0(\phi)(\chi_V(p)-\chi_V(q)) \\
&=& K_0(\phi)(\chi_V(p)) - K_0(\phi)(\chi_V(q)) \\
&=& \mathcal{H}(\chi_V(p)) - \mathcal{H}(\chi_V(q)) \\
&=& \mathcal{H}(\chi_V(p)-\chi_V(q)) \\
&=& \mathcal{H}([(p,q)])
\end{eqnarray*} 

for all $[(p,q)] \in K_0(V)$. Hence $K_0(\phi) = \mathcal{H}$.
\end{proof}

Let $V$ and $W$ be absolute matrix order unit spaces. We denote the zero map between $V$ and $W$ by $0_{W,V}$. Similarly, the identity map on $V$ is denoted by $I_V$. Further, if $V$ and $W$ satisfy (T), then we denote the zero group homomorphism between $K_0(V)$ and $K_0(W)$ by $0_{K_0(W),K_0(V)}$ and the identity map on $K_0(V)$ is denoted by $I_{K_0(V)}$.

\begin{corollary}\label{52}
Let $U,V$ and $W$ be absolute matrix order unit spaces satisfying (T). Then 
\begin{enumerate}
\item[(a)]$K_0(I_V)=I_{K_0(V)};$
\item[(b)]If $\phi: U \to V$ and $\psi: V \to W$ be unital completely $\vert \cdot \vert$-preserving maps, then $K_0(\psi \circ \phi) = K_0(\psi) \circ K_0(\phi);$
\item[(c)]$K_0(0_{W,V}) = 0_{K_0(W),K_0(V)}.$
\end{enumerate}
\end{corollary}

\begin{proof}
\begin{enumerate}
\item[(a)] Let $p,q \in \mathcal{OP}_\infty(V)$. Then
\begin{eqnarray*}
K_0(I_V)([(p,q)]) &=& [(I_V(p),I_V(q))] \\
&=& [(p,q)] 
\end{eqnarray*}

so that by Theorem \ref{51}, $K_0(I_V)=I_{K_0(V)}$.
\item[(b)] For any $[(p,q)] \in K_0(U)$, we get that
\begin{eqnarray*}
K_0(\psi \circ \phi)([(p,q)]) &=& [(\psi \circ \phi(p),\psi \circ \phi(q))] \\
&=& [(\psi(\phi(p)),\psi(\phi(q)))] \\
&=& K_0(\psi)[(\phi(p),\phi(q))] \\
&=& K_0(\psi)(K_0(\phi)([(p,q)])) \\
&=& K_0(\psi) \circ K_0(\phi)([(p,q)]). 
\end{eqnarray*}
Thus by Theorem \ref{51}, we conclude that $K_0(\psi \circ \phi) = K_0(\psi) \circ K_0(\phi)$.
\item[(c)] $K_0(0_{W,V})([(p,q)]) = [(0_{W,V}(p),0_{W,V}(q))]=[(0,0)]$ for all $[(p,q)] \in K_0(V)$. Thus again using \ref{51}, we get that $K_{0_{W,V}} = 0_{K_0(W),K_0(V)}$.
\end{enumerate}
\end{proof}

\begin{remark}
It follows from Corollary \ref{52} that $K_0$ is a functor from category of absolute matrix order unit spaces with morphisms as unital completely $\vert \cdot \vert$-preserving maps to category of abelian groups.
\end{remark}

\begin{corollary}\label{w49}
Let $V$ and $W$ be isomorphic absolute matrix order unit spaces (isomorphic in the sense that there exists a unital, bijective completely $\vert \cdot \vert$-preserving map between $V$ and $W$). Then $K_0(V)$ and $K_0(W)$ are group isomorphic. 
\end{corollary}

\begin{proof}
Let $\phi:V \to W$ be unital completely $\vert \cdot \vert$-preserving map. Then $\phi^{-1}$ is also unital completely $\vert \cdot \vert$-preserving map. Since $\phi^{-1} \circ \phi = I_{V}$ and $\phi \circ \phi^{-1} = I_{W}$, by Corollary \ref{52}(a) and (b), we get that $K_0(\phi^{-1}) \circ K_0(\phi) = I_{K_0(V)}$ and $K_0(\phi) \circ K_0(\phi^{-1}) = I_{K_0(W)}$. Thus $K_0(\phi):K_0(V) \to K_0(W)$ is a surjective group isomorphism and $K_0(\phi)^{-1} = K_0(\phi^{-1})$. Hence $K_0(V)$ and $K_0(W)$ are group isomorphic. 
\end{proof}

\subsection{Orthogonality of completely positive maps}

Now, we define orthogonality of completely positive linear maps between absolutely matrix ordered spaces.  

\begin{definition}\label{w51}
	Let $V$ and $W$ be absolute order unit spaces and let $\phi,\psi:V\to W$ be positive linear maps. We say that $\phi$ is \emph{orthogonal} to $\psi$ (we write, $\phi \perp \psi$), if $\phi(u) \perp \psi(v)$ for all $u,v \in V^+.$
\end{definition}

\begin{remark}\label{w50}
	Let $V$ and $W$ be absolute order unit spaces and let $\phi,\psi:V\to W$ be $\vert \cdot \vert$-preserving maps. Then $\phi \perp \psi$ if and only if $\phi(u) \perp  \psi(v)$ for all $u,v \in V$. Thus we get that $\vert \phi(u) \pm \psi(v)\vert = \vert \phi(u)\vert + \vert \psi(v)\vert$ for all $u,v \in V.$ In fact, if $\phi$ and $\psi$ are orthogonal $\vert \cdot \vert$-preserving maps, then by \cite[Proposition 2.6]{K19} and by Definition \ref{w51}, we have that $\phi(u)^+, \phi(u)^-, \psi(v)^+$ and $\psi(v)^-$ are mutually orthogonal. Thus by \cite[Proposition 2.5]{PI19}, we conclude that $\phi(u) \perp \psi(v)$ and $\vert \phi(u) \pm \psi(v) \vert = \vert \phi(u) \vert + \vert \psi(v) \vert.$ 
\end{remark}

\begin{definition}
	Let $V$ and $W$ be absolute matrix order unit spaces and let $\phi,\psi:V\to W$ be completely positive maps. We say that $\phi$ is \emph{completely orthogonal} to $\psi$ (we continue to write, $\phi \perp \psi$), if $\phi_n \perp \psi_n$ for all $n\in \mathbb{N}$. 
\end{definition}

In next theorem, we show that sum of two orthogonal completely $\vert \cdot \vert$-preserving maps is again a completely $\vert \cdot \vert$-preserving.

\begin{theorem}\label{w41}
	Let $V$ and $W$ be absolute matrix order unit spaces and $\phi,\psi:V\to W$ be completely $\vert \cdot \vert$-preserving maps such that $\phi \perp \psi$. Then $\phi+\psi$ is also completely $\vert \cdot \vert$-preserving map. 
\end{theorem}

\begin{proof}
	Let $v\in M_n(V)$. Then $\begin{bmatrix} 0 & v\\ v^* & 0\end{bmatrix} \in M_{2n}(V)_{sa}.$ Since $\phi$ is completely orthogonal to $\psi,$ by Remark \ref{w50}, we have $\phi_{2n}\left(\begin{bmatrix} 0 & v \\ v^* & 0 \end{bmatrix}\right) \perp \psi_{2n}\left(\begin{bmatrix} 0 & v \\ v^* & 0 \end{bmatrix}\right).$ Thus we get the following: 
	\begin{eqnarray*}
		\begin{bmatrix} \vert (\phi+\psi)_n(v^*)\vert_n & 0 \\ 0 & \vert (\phi+\psi)_n(v)\vert_n \end{bmatrix} &=& \left \vert \begin{bmatrix} 0 & (\phi+\psi)_n(v) \\ (\phi+\psi)_n(v^*) & 0 \end{bmatrix} \right \vert_{2n} \\
		&=& \left \vert \phi_{2n}\left(\begin{bmatrix} 0 & v \\ v^* & 0 \end{bmatrix}\right) + \psi_{2n}\left(\begin{bmatrix} 0 & v \\ v^* & 0 \end{bmatrix}\right) \right \vert_{2n} \\
		&=& \left \vert \phi_{2n}\left(\begin{bmatrix} 0 & v \\ v^* & 0 \end{bmatrix}\right)\right \vert_{2n} + \left \vert \psi_{2n}\left(\begin{bmatrix} 0 & v \\ v^* & 0 \end{bmatrix}\right) \right \vert_{2n}\\
		&=& \left \vert \begin{bmatrix} 0 & \phi_n(v)\\ {\phi_n(v^*)} & 0 \end{bmatrix}\right \vert_{2n}+\left \vert \begin{bmatrix} 0 & \psi_n(v)\\ {\psi_n(v^*)} & 0 \end{bmatrix}\right \vert_{2n}\\
		&=& \begin{bmatrix} \vert{\phi_n(v^*)}\vert_n & 0 \\ 0 & \vert \phi_n(v)\vert_n \end{bmatrix}+\begin{bmatrix} \vert {\psi_n(v^*)}\vert_n & 0 \\ 0 & \vert \psi_n(v)\vert_n \end{bmatrix}\\
		&=& \begin{bmatrix} \phi_n(\vert v^*\vert_n) & 0 \\ 0 & \phi_n(\vert v\vert_n) \end{bmatrix}+\begin{bmatrix} \psi_n(\vert v^*\vert_n) & 0 \\ 0 & \psi_n(\vert v\vert_n)\end{bmatrix} \\
		&=& \begin{bmatrix} (\phi+\psi)_n(\vert v^*\vert_n) & 0 \\ 0 & (\phi+\psi)_n(\vert v\vert_n)\end{bmatrix}
	\end{eqnarray*}
	
	so that $\vert (\phi+\psi)(v)\vert_n=(\phi+\psi)(\vert v\vert_n).$ Hence $\phi+\psi$ is completely $\vert \cdot \vert$-preserving.
\end{proof}

\begin{theorem}\label{w39}
	Let $V$ and $W$ be absolute matrix order unit spaces and $\phi,\psi:V\to W$ be completely $\vert \cdot \vert$-preserving maps such that $\phi \perp \psi$. If $\phi$ and $\psi$ map $\mathcal{OP}_\infty(V)$ into $\mathcal{OP}_\infty(W)$, then 
	\begin{enumerate}
		\item[(1)]$\phi+\psi$ also maps $\mathcal{OP}_\infty(V)$ into $\mathcal{OP}_\infty(W).$ 
		\item[(2)]$\phi,\psi$ and $\phi + \psi$ are completely contractive maps with $$\Vert \phi+\psi\Vert_{cb} = \textrm{max} \lbrace \Vert \phi\Vert_{cb}, \Vert \psi \Vert_{cb} \rbrace.$$   
		\item[(3)]If $V$ and $W$ satisfy (T), then $K_0(\phi+\psi)=K_0(\phi)+K_0(\psi).$ 
	\end{enumerate}
\end{theorem}

\begin{proof}
	Assume that $\phi$ and $\psi$ map $\mathcal{OP}_\infty(V)$ into $\mathcal{OP}_\infty(W).$
	\begin{enumerate}
		\item[(1)] Let $p\in \mathcal{OP}_n(V)$. Then $\phi_n(p),\psi_n(p)\in \mathcal{OP}_n(W)$ with $\phi_n(p)\perp \psi_n(p).$ By \cite[Proposition 3.2(1)]{PI19}, we get that $\phi_n(p)+\psi_n(p)\in \mathcal{OP}_n(V)$. Thus $\phi+\psi$ maps $\mathcal{OP}_\infty(V)$ into $\mathcal{OP}_\infty(W).$ 
		\item[(2)] Let $v \in M_n(V), n \in \mathbb{N}$. By \cite[Theorem 3.7(2)]{K19} and \cite[Theorem 3.8(1)]{K19}, we get that 
		\begin{eqnarray*}
			\Vert \phi_n(v) \Vert_n &=& \left \Vert \begin{bmatrix} 0 & \phi_n(v) \\ \phi_n(v)^* & 0\end{bmatrix}\right \Vert_{2n} \\
			&=& \left \Vert \phi_{2n}\left(\begin{bmatrix} 0 & v \\ v^* & 0\end{bmatrix}\right)\right \Vert_{2n} \\
			&\le & \left \Vert \begin{bmatrix} 0 & v \\ v^* & 0\end{bmatrix}\right\Vert_{2n} \\
			&=& \Vert v \Vert_n.
		\end{eqnarray*}
		Thus $\phi$ is completely contractive. Similarly, we can show that $\psi$ is also completely contractive. 
		
		Further, by Remark \ref{w50}, we have $$ \begin{bmatrix} 0 & \phi_{n}(v) \\ \phi_n(v)^* & 0 \end{bmatrix} = \phi_{2n} \left (\begin{bmatrix} 0 & v \\ v^* & 0 \end{bmatrix} \right) \perp \psi_{2n} \left (\begin{bmatrix} 0 & v \\ v^* & 0 \end{bmatrix} \right ) = \begin{bmatrix} 0 & \psi_{n}(v) \\ \psi_n(v)^* & 0 \end{bmatrix}$$ so that 
		
		\begin{eqnarray*}
			\Vert (\phi_n+\psi_n)(v)\Vert_n &=& \left\Vert \begin{bmatrix} 0 & (\phi_n+\psi_n)(v) \\ (\phi_n+\psi_n)(v)^* & 0\end{bmatrix}\right\Vert_{2n} \\ 
			&=& \left\Vert \left \vert \begin{bmatrix} 0 & (\phi_n+\psi_n)(v) \\ (\phi_n+\psi_n)(v)^* & 0\end{bmatrix}\right \vert_{2n} \right\Vert_{2n} \\
			&=& \left \Vert \left \vert \begin{bmatrix} 0 & \phi_n(v) \\ {\phi_n(v)}^* & 0\end{bmatrix} + \begin{bmatrix} 0 & \psi_n(v) \\ {\psi_n(v)}^* & 0\end{bmatrix}\right\vert_{2n}\right \Vert_{2n} \\
			&=& \left\Vert \left \vert \begin{bmatrix} 0 & \phi_n(v) \\ \phi_n(v)^* & 0\end{bmatrix}\right \vert_{2n} + \left \vert \begin{bmatrix} 0 & \psi_n(v) \\ \psi_n(v)^* & 0\end{bmatrix}\right \vert_{2n}\right\Vert_{2n} \\
			&=& \textrm{max}\left\lbrace\left\Vert \left \vert \begin{bmatrix} 0 & \phi_n(v) \\ \phi_n(v)^* & 0\end{bmatrix}\right \vert_{2n}\right\Vert_{2n}, \left\Vert\left \vert \begin{bmatrix} 0 & \psi_n(v) \\ \psi_n(v)^* & 0\end{bmatrix}\right \vert_{2n}\right\Vert_{2n} \right\rbrace \\
			&=& \textrm{max}\left\lbrace\left\Vert \begin{bmatrix} 0 & \phi_n(v) \\ \phi_n(v)^* & 0\end{bmatrix}\right\Vert_{2n}, \left\Vert\begin{bmatrix} 0 & \psi_n(v) \\ \psi_n(v)^* & 0\end{bmatrix}\right\Vert_{2n} \right\rbrace\\
			&=& \textrm{max} \left\lbrace \Vert\phi_n(v)\Vert_n , \Vert \psi_n(v) \Vert_n\right\rbrace.
		\end{eqnarray*}
		Thus $\Vert \phi+\psi\Vert_{cb} = \textrm{max} \lbrace \Vert \phi\Vert_{cb}, \Vert \psi \Vert_{cb} \rbrace$ and consequently $\phi + \psi$ is also completely contractive.
		
		\item[(3)]
		Assume that $V$ and $W$ satisfy (T). Let $p \in \mathcal{OP}_n(V)$ for some $n \in \mathbb{N}$. Since $\phi_n(p) \perp \psi_n(p)$, by \cite[Proposition 4.5(5)]{PI19}, we get that $\phi_n(p) + \psi_n(p)\sim \phi_n(p) \oplus \psi_n(p)$. Then 
		\begin{eqnarray*}
			K_0(\phi+\psi)([(p,0)]) &=& [((\phi+\psi)_n(p),0)]\\
			&=& [(\phi_n(p)+\psi_n(p),0)]\\
			&=& [(\phi_n(p) \oplus \psi_n(p),0)]\\
			&=& [(\phi_n(p),0)]+[(\psi_n(p),0)]\\
			&=& K_0(\phi)([(p,0)])+K_0(\psi)([(p,0)]).
		\end{eqnarray*}
		Thus by Theorem \ref{23}(c), we conclude that $K_0(\phi+\psi)=K_0(\phi)+K_0(\psi).$
	\end{enumerate}
\end{proof}

\begin{remark}
	Let $V$ and $W$ be absolute matrix order unit spaces and let $\phi_i:V \to W$ be completely $\vert \cdot \vert$-preserving maps for $i=1,2,\cdots,n$ such that $\phi_i \perp \phi_j$ for all $i \neq j.$ Then $\displaystyle\sum_{i=1}^n \phi_i$ is also a  completely $\vert \cdot \vert$-preserving map. Moreover, if $V$ and $W$ satisfy (T) and if $\phi$ maps $\mathcal{OP}_\infty(V)$ into $\mathcal{OP}_\infty(W)$ for each $i$, then $K_0\left(\displaystyle \sum_{i=1}^{n}\phi_i\right) = \displaystyle \sum_{i=1}^{n} K_0(\phi_i).$
\end{remark}

\end{document}